\documentclass[11pt,oneside]{article}
\usepackage{amsmath}
\usepackage{amssymb}
\usepackage[pdftex]{graphicx}
\usepackage{amsthm}
\usepackage[retainorgcmds]{IEEEtrantools}
\usepackage{subfigure}
\usepackage{multirow}
\usepackage{hyperref}

\newtheorem{theorem}{Theorem}

\newtheorem{lemma}{Lemma}
\newtheorem{remark}{Remark}
\newtheorem{assumption}{Assumption}
\numberwithin{equation}{section}
\numberwithin{figure}{section}

\title{Exponential Runge-Kutta schemes for inhomogeneous Boltzmann equations with high
order of accuracy\thanks{Research supported by Research Project of National Interest (PRIN 2009) \emph{Advanced numerical methods for kinetic equations and balance laws with source terms}.}}

\author{Qin Li\thanks{Department of Mathematics, University of Wisconsin-Madison, WI, USA}, Lorenzo
Pareschi\thanks{Department of Mathematics, University of Ferrara, Italy}}
\date{}

\begin{document}
\maketitle

\begin{abstract}
We consider the development of exponential methods for the robust time discretization of space inhomogeneous Boltzmann equations in stiff regimes. Compared to the space homogeneous case, or more in general to the case of splitting based methods, studied in Dimarco Pareschi \cite{DP_ExpRK} a major difficulty is that the local Maxwellian equilibrium state is not constant in a time step and thus needs a proper numerical treatment. We show how to derive asymptotic preserving (AP) schemes of arbitrary order and in particular using the Shu-Osher representation of Runge-Kutta methods we explore the monotonicity properties of such schemes, like strong stability preserving (SSP) and positivity preserving. Several numerical results confirm our analysis.

\end{abstract}

{\bf Keywords:} Exponential Runge-Kutta methods, stiff equations,
Boltzmann equation, fluid limits, asymptotic preserving schemes, strong stability preserving schemes.

\tableofcontents

\section{Introduction}
The time discretization of kinetic equations in stiff regimes represents a computational challenge in the construction of numerical methods. In fact, in regimes close to the fluid-dynamic limit the collisional scale becomes dominant over the transport of particles and forces the numerical methods to operate with time discretization steps of the order of the Knudsen number. On the other hand the use of implicit integration techniques presents considerable limitations in most applications since the cost required for the inversion of the collisional operator is prohibitive therefore limiting such techniques to simple linear operators.

In recent years there has been a remarkable development of numerical techniques specifically designed for such situations \cite{BLM, Filbet, toscani, DP_ExpRK, dpimex, Lem, PR}. The basic idea common to these techniques is to avoid the resolution of small time scales by using some a priori knowledge on the asymptotic behavior of the kinetic equation. In particular, we recall among the different possible approaches domain decomposition strategies and hybrid methods at different levels \cite{CPmc, degond1, dimarco1, dimarco2, TK}.

Asymptotic-preserving schemes have been particularly successful in the construction of unconditionally stable time discretization methods that avoids the inversion of the collision operator. For a nice survey on asymptotic-preserving scheme for various kinds of systems see, for example, the review paper by Shi Jin \cite{jinrev}. In the case of Boltzmann kinetic equations we also refer to the recent review by Pareschi and Russo \cite{prrev}.

In this paper we propose a new class of exponential integrators for the inhomogeneous Boltzmann
equation and related kinetic equations which is based on explicit exponential Runge-Kutta methods \cite{EX1, MZ}. More precisely we extend the method recently presented by one of the authors for homogeneous Boltzmann equations \cite{DP_ExpRK} to the inhomogeneous case by avoiding splitting techniques.
The main feature of the approach here proposed is that it works uniformly with very high-order for a wide range of
Knudsen numbers and avoids the solution of nonlinear systems of equations even in stiff regimes. Compared to penalized Implicit-Explicit (IMEX) techniques \cite{Filbet, dpimex} the main advantage of the class of methods here presented is the capability to easily achieve high order accuracy, asymptotic preservation and monotonicity of the numerical solution.

At variance with the approach presented in Dimarco, Pareschi \cite{DP_ExpRK} here we used the Shu-Osher representation of Runge-Kutta methods \cite{ShuOsher_ENO}. This turns out to be essential in order to obtain non splitting schemes with better monotonicity properties (usually referred to as strong stability properties \cite{GST}), which permits for example to obtain positivity preserving schemes.
In particular we construct methods which are uniformly accurate using two different strategies. The first class of methods is based on the use of a suitable time independent equilibrium state which permits to recover high order accuracy and positivity of the numerical solution. However since the method is based on a constant equilibrium computed at the end time it may suffer of accuracy deterioration in intermediate regimes. The second class of methods is based on computing explicitly the time variation of the Maxwellian state. This permits to obtain schemes with better uniform accuracy but loosing some of the monotonicity property obtained with the first technique.

The rest of the manuscript is organized as follows. In the next section we introduce some preliminary material concerning the Bolztmann equation and its fluid-limit. In Section 3 we derive the novel asymptotic-preserving exponential Runge-Kutta schemes. Two different approaches are presented. The properties of the two approaches are then studied in Section 4. In particular monotonicity properties are investigated. Finally in Section 5 several numerical results for schemes up to third order are presented which show the uniform high order accuracy properties of the present methods. 
Some theoretical proofs are reported in a separate appendix.

\section{The Boltzmann Equation and its fluid-dynamic limit}
\subsection{Boltzmann Equation}
The Boltzmann equation describes the evolution of the density distribution of
rarefied gases. We use $f(t,x,v)$ to represent the distribution function at
time $t$ on the phase space $(x,v)$. The Boltzmann equation is given by 
\begin{equation}\label{eqn_Boltzmann}
\partial_t f +v\cdot\nabla_x f=\frac1{\varepsilon}Q(f,f)\text{, }t\ge 0\text{,
}\left(x,v\right)\in \mathbb{R}^d\times\mathbb{R}^d,
\end{equation}
with
\begin{equation}\label{DefQ}
Q(f,f)=Q^+-fQ^-=\int_{S^{d-1}}\int_{\mathbb{R}^d} (f' f'_{*}-f f_{*}) B(|v-v_*|,{\omega})dv_* d{\omega}.
\end{equation}
Here, $B$ is the collision kernel, $\varepsilon>0$ is the Knudsen number, $\omega$ is a unit vector, and $S^{d-1}$ is the unit
sphere defined in $R^d$ space.
We use the shorthands $f'=f(t,x,v')$ and
$f'_{*}=f(t,x,v'_*)$.
There are many variations for the collision kernel $B$. One simple
case is the case of Maxwell molecules when 
$$B=B\left(\frac{g\cdot\omega}{|g|}\right),$$ with the relative velocity 
$g=v-v_*$.\\
The collisional velocities $v'$ and $v'_*$ satisfy
\begin{subequations}\label{VelocityAfterCollision}
    \begin{align}
        v'&=v-\frac{1}{2}(g-|g|\omega),\\
      v_*'&=v_*+\frac{1}{2}(g-|g|\omega).
    \end{align}
\end{subequations}
This deduction is based on momentum and energy conservations
\begin{align*}
v+v_*&=v'+v'_{*},\\
|v|^2+|v_*|^2&=|v'|^2+|v'_{*}|^2.
\end{align*}

In $d$-dimensional space, we define the following macroscopic quantities 
$\rho$ is the mass density (here we assume mass is 1, thus number density
and mass density have the same value); $u$ is a $d$-dimensional vector that
represent the average velocity; $E$ is the total energy; $e$ is the
specific internal energy; $T$ is the temperature; $S$ is the stress tensor; and $q$ is the heat flux vector, given by 
\begin{align}
    &\rho=\int f dv,
&\rho {u}=\int {v}f dv,\nonumber\\
&E=\frac{1}{2}\rho u^2+\rho e=\frac{1}{2}\int |{v}|^2fdv,
&e=\frac{d}{2}T=\frac{1}{2\rho}\int f|{v}-{u}|^2dv,\\
&S=\int ({v-u})\otimes({v-u})fdv,
&{q}=\frac{1}{2}\int ({v}-{u})|{v}-{u}|^2fdv.\nonumber
\end{align}

\subsection{Conservations and fluid limit}
Cross section may vary, but the first $d+2$ moments of the collision term are always
zero. They are obtained by multiplying the collision term with
$\phi=\left(1,v,\frac{1}{2}|v|^2\right)^T$ and then integrating with respect
to $v$, i.e. 
\begin{align}\label{eqn_momentsMaxwell}
<Q>&=\int Q(f)dv=0,\nonumber\\
<vQ>&=\int {v} Q(f) dv=0,\nonumber\\
<\frac{1}{2}v^2Q>&=\int \frac{1}{2}|v|^2Q(f)d{v}=0.
\end{align}
Based on these formulas, when taking moments of the Boltzmann equation,
one obtains mass, momentum and energy conservation 
\begin{align}\label{MacroEvolveSingle}
& \partial_t \rho+\nabla_x\cdot(\rho{u})=<Q>=0,\nonumber\\
& \partial_t (\rho {u})+\nabla_x\cdot(S+\rho {u}^2)=\frac{1}{\varepsilon}<vQ>=0,\nonumber\\
&\partial_t E+\nabla_x\cdot(E{u}+S{u}+q)=\frac{1}{\varepsilon}<\frac{1}{2}|v|^2Q>=0.\nonumber
\end{align}

For small values of $\varepsilon$, the standard
Chapman-Enskog expansion around the local Maxwellian
\begin{equation}\label{def_Maxwellian}
    M(t,x,v)=\rho(t,x)\left(\frac{1}{2\pi T(t,x)}\right)^{d/2}\exp{\left(-\frac{(v-u(t,x))^2}{2T(t,x)}\right)},
\end{equation}
shows that at the leading order the moment system yields its Euler limit
\begin{align}\label{EulerLimit}
& \partial_t \rho+\nabla\cdot (\rho{u})=0,\nonumber\\
& \partial_t (\rho {u})+\nabla \cdot(\rho {u}\otimes {u}+\rho {T}\mathbb{I})=0,\\
&\partial_t E+\nabla \cdot((E+\rho {T}){u})=0,\nonumber
\end{align}
where $\mathbb{I}$ is the identity matrix.

\section{Exponential Runge-Kutta (ExpRK) methods}
In this section we would like to extend the Exponential RK method in 
\cite{DP_ExpRK} for the homogeneous Boltzmann equation to the inhomogeneous 
case (\ref{eqn_Boltzmann}). It has been known for long that time splitting 
methods degenerate to first order accuracy in the fluid-limit (see \cite{DP_ExpRK} and the references therein) so, to achieve 
high order of accuracy in stiff regimes, time splitting should be avoided.

\subsection{Reformulation of the problem and notations}
To achieve AP property and robustness in stiff regimes, an implicit method should be adopted. However, due to the
complexity and nonlocal property of the collision term $Q$, directly inverting
it is prohibitively expensive. The Exponential Runge-Kutta method overcomes this
difficulty by transforming the equation into the exponential form, and forces
the solution to approach to the equilibrium that captures its asymptotic
Euler limit as $\varepsilon$ tends to zero, thus it is an AP scheme. Following
the approach in \cite{DP_ExpRK}, one can define
\begin{align}
    P=Q+\mu f\label{def_P}, \hspace{0.5cm}\mu>0.
\end{align}
Let us now consider a nonnegative function $\tilde M$, hereafter called the \emph{equilibrium function}, and using (\ref{eqn_Boltzmann}) compute 
\begin{eqnarray}\label{eqn_FullExpForm}
    &&\partial_t \left[(f-\tilde{M})e^{\mu t/\varepsilon}\right]\nonumber\\
    & = &\partial_t (f-\tilde{M}) e^{\mu t/\varepsilon}+(f-\tilde{M})\frac{\mu}{\varepsilon}e^{\mu t/\varepsilon}\nonumber\\
    & = &\left[\frac{1}{\varepsilon}(Q+\mu f-\mu \tilde{M})-\partial_t\tilde{M}-v\cdot\nabla_xf\right]e^{\mu t/\varepsilon}\\\nonumber.
    & = &\left[\frac{1}{\varepsilon}(P-\mu \tilde{M})-\partial_t\tilde{M}-v\cdot\nabla_xf\right]e^{\mu t/\varepsilon}\nonumber.
\end{eqnarray}
Note that the equation above is equivalent to the original Boltzmann equation 
(\ref{eqn_Boltzmann}) as long as $\mu$ is independent on time. In the simplified case of the BGK collision operator $Q=\mu(M-f)$, where $M$ is the local Maxwellian given by (\ref{def_Maxwellian}), the problem reformulation just described applies with $P=\mu M$. Moreover  
there is no requirement on the form of $\tilde{M}$ at all -- it can be an 
arbitrary function. However, to obtain AP property, one has to be careful in 
picking up its definition, so that the correct
asymptotic limit could be captured.

We analyze two different approaches in the following
two subsections, and adopt a suitable explicit Runge-Kutta scheme to solve them.
For readers' convenience, we firstly give the expression of the Runge-Kutta method used
here. Given a large set of ODEs
\begin{equation}\label{eqn_classicalRK}
    \partial_t y=F(t,y),
\end{equation}
obtained for example using the method of lines from a given PDE, 
if data $y^n$ at time step $t^n$ is known, to compute for the value $y^{n+1}$
at $t^{n+1}=t^n+h$, a classical $\nu$-step explicit Runga-Kutta scheme for equation (\ref{eqn_classicalRK})
writes 
\begin{equation}
\begin{cases}
    \text{Step $i$:}\hspace{1cm}&\displaystyle y^{n,(i)}=y^{n}+h\sum^{i-1}_{j=1}a_{ij}F(t^n+c_jh,y^{n,(j)}),\label{scheme_ExpRKk_classic}\\
    \text{Final step:}\hspace{1cm}&\displaystyle y^{n+1}=y^{n}+h\sum^\nu_ib_iF(t^n+c_ih,y^{n,(i)}),
\end{cases}
\end{equation}
where $\sum_{j=1}^{i-1}a_{ij}=c_i$, $\sum_ib_i=1$, and $y^{n,(i)}$ stands for the
estimate of $y$ at $t=t^n+c_ih$. Different Runge-Kutta method gives
different set of coefficients. In the sequel we drop superscript $n$ for 
evaluation of $y$ at sub-stages and use $y^{(i)}=y^{n,(i)}$.

Another form of RK method which has proved to be useful in the analysis of the monotonicity properties of Runge-Kutta schemes is the so-called {\em Shu-Osher representation}
\cite{ShuOsher_ENO}. This representation is essential in the study of the positivity properties that will be carried out later
\begin{equation}\label{scheme_RKPosForm}
\begin{cases}
    \text{Step $i$:}\hspace{1cm} &\displaystyle y^{(i)}=\sum_{j=1}^{i-1}\left[\alpha_{ij}y^{(j)}+h\beta_{ij}F(t^n+c_jh,y^{(j)})\right],\\
    \text{Final step:}\hspace{1cm} &\displaystyle 
    y^{n+1}=\sum_{j=1}^{\nu}\left[\alpha_{\nu+1 j}y^{(j)}+h\beta_{\nu+1 
    j}F(t^n+c_jh,y^{(j)})\right].\\
\end{cases}
\end{equation}
Let us point out that this latter representation is not unique. Here 
$\alpha_{ij}$ are parameters such that $\sum_{j=1}^{i-1}\alpha_{ij}=1$. 
Without loss of generality, it is natural to set
\begin{eqnarray}
\beta_{ij}=\alpha_{ij}\left(c_i-c_j\right),
\label{eq:relations}
\end{eqnarray}
for consistency.
\begin{remark}
    Expression (\ref{eq:relations}) is equivalent with the classical one which says \cite{ShuOsher_ENO}
    \begin{equation}\beta_{ij}=a_{ij}-\sum_{k=j+1}^{i-1}\alpha_{ik}a_{kj}.\end{equation} In fact, 
    assume one has $y^{(j)}=y^n+h\sum_{k=1}^{j-1}a_{jk}F^{(k)}$, $\forall\,\, j<i$, where $F^{(k)}$ is a shorthand for $F(t^n+c_kh,y^{(k)})$, 
    then, by (\ref{eq:relations}) one has
    \begin{align}
        y^{(i)}&=\sum_{j=1}^{i-1}\left[\alpha_{ij}y^{(j)}+\alpha_{ij}(c_i-c_j)hF^{(j)}\right]\nonumber\\
               &=\sum_{j<i}\left[\alpha_{ij}\left(y^n+h\sum_{k<j}a_{jk}F^{(k)}\right)+\alpha_{ij}(c_i-c_j)hF^{(j)}\right]\nonumber\\
               &=y^n+h\sum_{j<i}\left(\sum_{k=j+1}^{i-1}\alpha_{ik}a_{kj}+\alpha_{ij}(c_i-c_j)\right)F^{(j)}
    \end{align}
    This clearly requires 
    $a_{ij}=\alpha_{ij}(c_i-c_j)+\sum\alpha_{ik}a_{kj}$. Given (\ref{eq:relations}), it 
    is $a_{ij}=\beta_{ij}+\sum\alpha_{ik}a_{kj}$, which is exactly the 
    classical Shu-Osher representation.
\end{remark}

\subsection{Exponential RK schemes with fixed equilibrium function}
Since the choice of the equilibrium function $\tilde{M}$ in (\ref{eqn_FullExpForm}) is arbitrary, in this subsection, we assume $\tilde{M}$ as a function independent of time in
each time step, i.e. $\tilde{M}$ is a function given a-priori. Thus
(\ref{eqn_FullExpForm}) could be rewritten as 
\begin{eqnarray}\label{eqn_fixedMExpForm}
    \partial_t \left[(f-\tilde{M})e^{\mu t/\varepsilon}\right]=\left[\frac{1}{\varepsilon}(P-\mu \tilde{M})-v\cdot\nabla_xf\right]e^{\mu t/\varepsilon}.
\end{eqnarray}
    Analytically, the equation (\ref{eqn_fixedMExpForm}) is equivalent to
    the original inhomogeneous Boltzmann equation as long as $\tilde{M}$ is a function independent of time and $\mu$ is a constant.
    But the associated numerical scheme can preserve asymptotic limit only if
    $\tilde{M}$ is chosen in a correct way, as will be clearer later. On the other hand $\mu$ plays a role in order to guarantee positivity of the numerical solution as will be seen in section. 
\begin{remark}
    Obviously $\tilde{M}$ is required not to change in each time step. But for
    different time steps, we are free to use different functions. This is in
    fact what we will do, we evolve $\tilde{M}$ before each time step with a suitable scheme, and then use
    this computed value function to construct the AP exponential scheme.
\end{remark}
\subsubsection{The numerical scheme: ExpRK-F}
Compared to (\ref{eqn_classicalRK}), $y$ turns out to be
$(f-\tilde{M})e^{\mu t/\varepsilon}$ and the associated evolution function $F(t,y)$
on the right of (\ref{eqn_classicalRK}) is
$\left[\frac{1}{\varepsilon}(P-\mu \tilde{M})-v\cdot\nabla_x f\right]e^{\mu t/\varepsilon}$
. Thus we have the following scheme 
\begin{equation}\label{scheme_ExpRK}
    \begin{cases}
        \text{Step $i$:}\hspace{1cm} &(f^{(i)}-\tilde{M})e^{c_i\lambda}=(f^{n}-\tilde{M})+\displaystyle\sum^{i-1}_{j=1}a_{ij}\frac{h}{\varepsilon}\left[ P^{(j)}-\mu \tilde{M} -\varepsilon v\cdot\nabla_xf^{(j)}\right]e^{c_j\lambda},\\
    \text{Final Step:}\hspace{.5cm} &(f^{n+1}-\tilde{M})e^{\lambda}=(f^{n}-\tilde{M})+\displaystyle\sum_{i=1}^\nu b_i\frac{h}{\varepsilon}\left[ P^{(i)}-\mu \tilde{M} -\varepsilon v\cdot\nabla_xf^{(i)}\right]e^{c_{i}\lambda}.
    \end{cases}
\end{equation}
where we used $\lambda=\frac{\mu h}{\varepsilon}$, and $P^{(j)}=P(f^{(j)})$
for simplicity. Simple algebra gives 
\begin{itemize}\label{scheme_ExpRKFinal}
    \item{Step $i$:}\begin{equation*}
            f^{(i)}=\left(1-e^{-c_i\lambda}-\sum_{j=1}^{i-1}a_{ij}\lambda e^{\lambda(-c_i+c_j)}\right)\tilde{M}+e^{-c_i\lambda}f^{n}+\sum_{j=1}^{i-1}a_{ij}\lambda e^{\lambda(c_j-c_i)}\left(\frac{P^{(j)}}{\mu}-\frac{\varepsilon}{\mu}v\cdot\nabla_xf^{(j)}\right),
    \end{equation*}
\item{Final Step:}
    \begin{equation*}
        f^{n+1}=\left(1-e^{-\lambda}-\sum_ib_i\lambda e^{\lambda(-1+c_i)}\right)\tilde{M}+e^{-\lambda}f^{n}+\sum_ib_i\lambda e^{\lambda(c_i-1)}\left(\frac{P^{(i)}}{\mu}-\frac{\varepsilon}{\mu}v\cdot\nabla_xf^{(i)}\right).
\end{equation*}
\end{itemize}

\subsubsection{Choice and evaluation of $\tilde{M}$}
If it is assumed that
\begin{equation}
0=c_1<c_2<\ldots < c_\nu < 1,
\label{eq:assum}
\end{equation}
then the same arguments used in \cite{DP_ExpRK} shows immediately, that as $\varepsilon\to 0$, $\lambda\to\infty$, the
scheme pushes $f^{n+1}$ going to $\tilde{M}$. So to obtain AP property, $\tilde{M}$
above should be the Maxwellian at time level $n+1$ that has the right moments. To get the right moments,
the simplest way is to evolve the corresponding macroscopic limit equations, say
the Euler equation. We propose solving the Euler equation first to obtain the
macroscopic quantities of the Maxwellian for the next time step, and make use of them to
define $\tilde{M}$. To achieve high order for all regimes, both the
macro-solver and micro-solver should be handled by numerical schemes with the 
same order of accuracy in space and time. The most natural way in time 
discretization is the explicit Runge-Kutta scheme using the same coefficients 
as the one for the kinetic equation 
    \begin{equation}
        \begin{cases}
            \text{Step }i:\hspace{0.5cm}&\left(\begin{array}{c}\rho\\\rho u\\E\end{array}\right)^{(i)}=
        \left(\begin{array}{c}\rho\\\rho u\\E\end{array}\right)^n-\Delta
        t\displaystyle\sum_{j=1}^{i-1}a_{ij}\nabla_x\cdot\left(\begin{array}{c}\rho
            u\\\rho u\otimes u+\rho T\\\left(E+\rho T\right)u\end{array}\right)^{(j)},\\
            \text{Final Step:}\hspace{0.5cm}&\left(\begin{array}{c}\rho\\\rho u\\E\end{array}\right)^{n+1}=
        \left(\begin{array}{c}\rho\\\rho u\\E\end{array}\right)^n-\Delta
        t\displaystyle\sum_{i=1}^{\nu}b_i\nabla_x\cdot\left(\begin{array}{c}\rho
            u\\\rho u\otimes u+\rho T\\\left(E+\rho T\right)u\end{array}\right)^{(i)}.
        \end{cases}
    \end{equation}
\begin{remark}~
\begin{itemize}
\item Note that this method gives us a simple way to couple 
    macro-solver with micro-solver. When $\varepsilon$ is considerably big, 
    the accuracy of the method is controlled by the micro-solver. And as 
    $\varepsilon$ vanishes, the method pushes $f$ going to $M$, which is 
    defined by macroscopic quantities computed through the Euler equation 
    while the order of accuracy is given by the macro-solver.
\item In principle it is possible to adopt other strategies to compute a more 
    accurate time independent equilibrium function in intermediate regions. 
    For example one can use the $ES-BGK$ Maxwellian \cite{Filbet3} at time 
    $n+1$ or one can use the Navier-Stokes equation as the macro-counterpart. 
    Here however we do not explore further in these directions.
\item The assumption (\ref{eq:assum}), although strongly simplifies 
    computations, is in fact not necessary to prove asymptotic 
    preservation. In fact such assumption is independent of the structure of 
    the operator $P(f,f)$. We refer to Section 4.2 for more details.  
\end{itemize}
\end{remark}

\subsection{Exponential Runge-Kutta schemes with time varying equilibrium function}
The approach just described has the nice feature of being extremely simple to 
construct and implement. As we will see in the next section it also possesses 
several nice features concerning monotonicity. On the other hand it is clear 
that choosing the limiting equilibrium state in the construction may produce a 
lack of accuracy in intermediate regimes. To overcome this aspect here we 
consider the most natural choice of equilibrium function, namely the local 
Maxwellian equilibrium state $\tilde M=M$. The major difficulty in this case 
is due to the time dependent nature of such equilibrium function.

Now rewrite the equation as
\begin{align}\label{eqn_FullExpForm-V}
    \partial_t\left[\left(f-M\right)\exp{\left(\frac{\mu
t}{\varepsilon}\right)}\right]=\left(\frac{P-\mu
    M}{\varepsilon}-v\cdot\nabla_xf-\partial_tM\right)\exp{\left(\frac{\mu t}{\varepsilon}\right)},
\end{align}
and here we define $\tilde{M}$ has a Gaussian profile that shares the same
first $d+2$ moments with $f$. The moments'
equations are governed by 
\begin{align}\label{eqn_Exp2Moment}
    \partial_t\int\phi f dv+\int\phi v\cdot\nabla_xfdv=0,
\end{align}
with $\phi=\left[1,v,\frac{v^2}{2}\right]^T$.
\subsubsection{The numerical scheme: ExpRK-V}
The Runge-Kutta method is adopted to solve the system 
\begin{align*}
    \begin{cases}
    \displaystyle\partial_t (f-{M})e^{\mu t/\varepsilon}&= \displaystyle\frac{1}{\varepsilon}(P-\mu {M}-\varepsilon v\cdot\nabla_xf-\varepsilon\partial_tM)e^{\mu t/\varepsilon},\\[+.2cm]
    \displaystyle\partial_t\int\phi fdv&=-\int\phi v\cdot\nabla_x fdv.
    \end{cases}
\end{align*}
Thus we have the following scheme 
\begin{subequations}\label{scheme_M2}
\begin{description}
        \item[]{Step $i$:}
            \begin{equation}\label{scheme_M2StepK}\begin{cases}
                (f^{(i)}-M^{(i)})e^{c_i\lambda}& =(f^{n}-M^n)+\displaystyle\sum^{i-1}_{j=1}a_{ij}\frac{h}{\varepsilon}\left[ P^{(j)}-\mu M^{(j)} -\varepsilon v\cdot\nabla_xf^{(j)}-\varepsilon\partial_tM^{(j)}\right]e^{c_{j}\lambda},\\
                           \int \phi f^{(i)}dv&  =\int\phi f^ndv+\displaystyle\sum^{i-1}_{j=1}a_{ij}\left(-h\int\phi v\cdot\nabla_xf^{(j)}dv\right);\end{cases}
        \end{equation}
        \item[]{Final Step:}
            \begin{equation}\label{scheme_M2Final}\begin{cases}
                    (f^{n+1}-M^{n+1})e^{\lambda}&=(f^{n}-M^n)+\displaystyle\sum_{i=1}^\nu b_i\frac{h}{\varepsilon}\left[ P^{(i)}-\mu M^{(i)} -\varepsilon v\cdot\nabla_xf^{(i)}-\varepsilon\partial_tM^{(i)}\right]e^{c_{i}\lambda},\\
                              \int\phi f^{n+1}dv&=\int\phi f^ndv+\displaystyle\sum_{i=1}^\nu b_i\left(-h\int\phi v\cdot\nabla_xf^{(i)}dv\right).\end{cases}
            \end{equation}
\end{description}
\end{subequations}
The first equation in (\ref{scheme_M2StepK}) shows that in each sub-stage $i$,
to compute for $f^{(i)}$, besides the known $f^{(j)}$ and easily obtained
$M^{(j)}$, one also needs $\partial_t M^{(j)}$, $P^{(j)}$,
$v\cdot\nabla_xf^{(j)}$ for all $j<i$, and $M^{(i)}$ that is evaluated at the
current time sub-stage. 

\subsubsection{Computation of $M$ and $\partial_t M$}
Here we show how to compute $M^{(i)}$ and $\partial_t M^{(j)}$.

\begin{description}
    \item[Computation of $M^{(i)}$]{}:\\
solve the second equation of (\ref{scheme_M2StepK}), to get evaluation of
macroscopic quantities at $t^n+c_ih$. Then the Maxwellian $M^{(i)}$ is given
by (\ref{def_Maxwellian}).
\item[Computation of $\partial_t M^{(j)}$]{}:\\
Write $\partial_tM$ as 
\begin{align}
    \partial_tM=\partial_\rho M\partial_t\rho+\nabla_uM\cdot\partial_tu+\partial_TM\partial_tT,
\end{align}
and $\partial_t\rho$, $\partial_tu$ and $\partial_tT$ can be computed from taking moments
of the original equation 
\begin{eqnarray}
\nonumber
    \partial_t\left(\begin{array}{c}
        \rho\\
        \rho u\\
        \frac{d\rho T}{2}+\frac{1}{2}\rho u^2\end{array}\right)&=&\partial_t\int\left(\begin{array}{c}
            1\\
            v\\
            \frac{v^2}{2}\end{array}\right) Mdv=\partial_t\int\left(\begin{array}{c}1\\v\\\frac{v^2}{2}\end{array}\right) fdv\\[-.5cm]            
            \\
            \nonumber
            &=&-\int \left(\begin{array}{c}1\\v\\\frac{v^2}{2}\end{array}\right) v\cdot\nabla_x fdv.\label{scheme_ptM}
\end{eqnarray}
To be specific, with data at sub-stage $(j)$ in $d$-dimensional space, one has 
\begin{align}
    \partial_tM^{(j)}=\partial_\rho M^{(j)}\partial_t\rho^{(j)}+\nabla_uM^{(j)}\cdot\partial_tu^{(j)}+\partial_TM^{(j)}\partial_tT^{(j)},
\end{align}
with
\begin{subequations}
\begin{equation}
    \partial_\rho M^{(j)}=\frac{M^{(j)}}{\rho^{(j)}},\,\partial_uM^{(j)}=M^{(j)}\frac{v-u^{(j)}}{T^{(j)}},\,\partial_TM^{(j)}=M^{(j)}\left[\frac{(v-u^{(j)})^2}{2(T^{(j)})^2}-\frac{d}{2T^{(j)}}\right],
\end{equation}
and
\begin{align}
    \partial_t\rho^{(j)}&=-\int v\cdot\nabla_xf^{(j)}dv,\label{scheme_ptRho}\\
       \partial_tu^{(j)}&=\frac{1}{\rho^{(j)}}\left(u^{(j)}\int v\cdot\nabla_xf^{(j)}dv-\int v\otimes v\cdot\nabla_xf^{(j)}dv\right),\label{scheme_ptU}\\
       \partial_tT^{(j)}&=\frac{1}{d\rho^{(j)}}\left(-\frac{2E^{(j)}}{\rho^{(j)}}\partial_t\rho^{(j)}-2\rho^{(j)} u^{(j)}\partial_tu^{(j)}-\int v^2 v\cdot\nabla_xf^{(j)}dv\right).\label{scheme_ptT}
\end{align}
\end{subequations}
The $\partial_t\rho$ and $\partial_t u$ term in (\ref{scheme_ptT}) is evaluated by
(\ref{scheme_ptRho}) and (\ref{scheme_ptU}). Clearly, all other macroscopic
quantities $\rho^{(j)}$, $u^{(j)}$ and $T^{(j)}$ are associated to $f^{(j)}$.
\end{description}

\section{Properties of ExpRK schemes}
\subsection{Positivity and monotonicity properties}
Usually positivity, although very important for kinetic equations, is extremely hard to be obtained when using high order schemes. Here we show that thanks to the Shu-Osher representation
(\ref{scheme_RKPosForm}) we can follow
\cite{GottliebShu_TVDRK} to prove positivity (and hence SSP property) for the fixed $\tilde{M}$ method ExpRK-F. 

Before proving the theorem we make the following assumption.
\begin{assumption}
For a given $f \geq 0$ there exists $h^*>0$ such that
\begin{equation*}
    f-h\, v\cdot\nabla_xf \geq 0, \hspace{0.5cm}\forall\,\, 0< h \leq h^*.
\end{equation*}
\end{assumption}
The above assumption is the minimal requirement on $f$ in order to obtain a non negative scheme.
Next we can state 
\begin{theorem}
    Let us consider an ExpRK-F method defined by (\ref{scheme_ExpRK}), and 
    $\beta_{ij} \geq 0$ in (\ref{eq:relations}). Then there exist $h_*>0$ and 
    $\mu_*>0$ such that $f^{n+1} \geq 0$ provided that $f^n\geq 0$, 
    $\mu\geq \mu_*$ and $0<h\leq h_*$.
\end{theorem}

\begin{proof}
Using the Shu-Osher representation, one could rewrite the scheme as 
\begin{equation*}
    \begin{cases}
        \text{Step $i$:}\hspace{0.7cm} &(f^{(i)}-\tilde{M})e^{c_i\lambda}=\displaystyle\sum_je^{c_j\lambda}\left\{\alpha_{ij}(f^{(j)}-\tilde{M})+\beta_{ij}\frac{h}{\varepsilon}\left[ P^{(j)}-\mu \tilde{M} -\varepsilon v\cdot\nabla_xf^{(j)}\right]\right\}\\
     \text{Final Step:}\hspace{0.2cm}
     &(f^{n+1}-M)e^{\lambda}=\displaystyle\sum_je^{c_j\lambda}\left\{\alpha_{\nu+1 j}(f^{j}-\tilde{M})+\beta_{\nu+1 j} \frac{h}{\varepsilon}\left[ P^{(j)}-\mu \tilde{M} -\varepsilon v\cdot\nabla_xf^{(j)}\right]\right\}
    \end{cases}
\end{equation*}
Simple algebra gives, for $\forall$, $i=1,\cdots,\nu$, $j<i$ 
\begin{align}\label{scheme_ExpRKPosForm}
    f^{(i)}=&\tilde{M}\left(1-\sum_je^{(c_j-c_i)\lambda}\left(\alpha_{ij}+\lambda\beta_{ij}\right)\right)\nonumber\\
            &+\sum_{j=1}^{i-1}\lambda\beta_{ij}e^{(c_j-c_i)\lambda}\frac{P^{(j)}}{\varepsilon}\nonumber\\
            &+\sum_{j=1}^{i-1}\alpha_{ij}e^{(c_j-c_i)\lambda}\left(f^{(j)}-\frac{h\beta_{ij}}{\alpha_{ij}}v\cdot\nabla_xf^{(j)}\right).
\end{align}
The same derivation can be also carried out for the final step. If 
this is a convex combination, then, to have positivity, one only 
check that each of them is positive
\begin{subequations}
\begin{align}
    &\tilde{M}> 0;\\
    &P^{(j)}>0;\\
    &f^{(j)}-\frac{h\beta_{ij}}{\alpha_{ij}}v\cdot\nabla_xf^{(j)}>0.\label{cond_transport}
\end{align}
\end{subequations}
Positivity of $\tilde{M}$ is obvious, and $P^{(j)}$ is positive if one has big 
enough $\mu$
\begin{equation*}
    \mu\geq\mu_*=\sup{|Q^-|}\Rightarrow P=Q+\mu f=Q^+-fQ^-+\mu f>0.
\end{equation*}
To handle (\ref{cond_transport}), one just need to adopt {Assumption 
1}. It is positive if
\begin{equation*}
    0<h\leq h_*=\min_{ij}{\left(\frac{\alpha_{ij}}{\beta_{ij}}h^*\right)},
\end{equation*}
which guarantees (\ref{cond_transport}).

To check the convexity of (\ref{scheme_ExpRKPosForm}), it should be proved that
\begin{equation}\label{scheme_PosConvex}
\sum_{j}e^{(c_j-c_i)\lambda}\left(\alpha_{ij}+\lambda\beta_{ij}\right)\leq 1.
\end{equation}
This can be seen by just taking the derivative with respect to $\lambda$. Use 
$\Delta_{ij}=c_i-c_j$
\begin{align}
    &\frac{d}{d\lambda}\left(\sum_je^{-\Delta_{ij}\lambda}(\alpha_{ij}+\lambda\beta_{ij})\right)\nonumber\\
   =&\sum_{j}e^{-\Delta_{ij}\lambda}\left(-\Delta_{ij}(\alpha_{ij}+\lambda\beta_{ij})+\beta_{ij}\right)\\
   =&\sum_{j}e^{-\Delta_{ij}\lambda}\left(-\beta_{ij}\Delta_{ij}\lambda+\beta_{ij}-\alpha_{ij}\Delta_{ij}\right)<0
\end{align}
In the last step, $\beta_{ij}=\alpha_{ij}(c_i-c_j)$ is used. Thus the 
left-hand side of expression (\ref{scheme_PosConvex}) is monotonically 
decreasing with respect to $\lambda\geq 0$ and has a maximum
$$
\sum_{j}\alpha_{ij}=1
$$
at $\lambda=0$. Similarly we can proceed for the final step. This confirms 
(\ref{scheme_PosConvex}) and finishes our proof.\end{proof}
Since the proof above is based on a convexity argument, we also have 
monotonicity of the numerical solution or SSP property. Thus the building 
block of our exponential schemes is naturally given by the optimal SSP schemes 
which minimize the stability restriction on the time stepping. We refer to 
\cite{GST} for a review on SSP Runge-Kutta schemes.
 
\begin{remark}~
\begin{itemize}
    \item Note that the proof above does not rely on the value $\lambda$ take, 
        i.e. the scheme is positive uniformly in $\varepsilon$. For the choice of $\mu_*$ we refer the reader to the discussion in \cite{DP_ExpRK, toscani}.        
        
\item Optimal second and third order SSP explicit Runge-Kutta methods such 
    that $\beta_{ij}\geq 0$ have been developed in the literature. However the 
    classical third order SSP method by Shu and Osher \cite{ShuOsher_ENO} does 
    not satisfy $c_j\leq c_i$ for $j < i$. Note that standard second order 
    midpoint and third order Heun methods satisfy the assumptions of
    Theorem 1 (see Table 1.1 page 135 in \cite{HNW}).

\item In \cite{GottliebShu_TVDRK} it was proved that allfour stage, fourth order 
RK methods with positive CFL coefficient $h_*$ must have at least one negative 
$\beta_{ij}$. The most popular fourth order method using five stage with 
nonnegative $\beta_{ij}$ has been developed in \cite{RuSp}. In \cite{RuSp} the 
authors also proved that any method of order greater then four will have 
negative $\beta_{ij}$.

\item    Positivity of ExpRK-V schemes is much more difficult to achieve because of the
    involvement of the $\partial_tM$ term. However, we can prove:
    \begin{itemize}
        \item[(i)] $\rho$ is positive;
        \item[(ii)] the negative part of $T$ is $O(h\varepsilon)$.
     \end{itemize}
        We leave both the proofs of the above results to
        the appendix.
\end{itemize}
\end{remark}

\subsection{Contraction and Asymptotic Preservation}
In this section, it will be presented that the new exponential Runge-Kutta schemes preserve the asymptotic
limit of the Boltzmann equation. The proof is done by following the proof of
contraction in \cite{DP_ExpRK}.

If one check the formula (\ref{scheme_ExpRK}) and (\ref{scheme_M2Final}), it
seems clear that under assumptions (\ref{eq:assum}) the big $\lambda$ on the shoulder of exponential will push
the distance between $f$ and the Maxwellian function going to zero. But sometimes
the Runge-Kutta method may have tough coefficients, say $c_\nu=1$. When this
happens, the argument cannot be carried through. However, one could still
prove AP property using the particular structure of the collision operator 
following the framework below.

We need to make use of the following assumption.
\begin{assumption}\label{assumption_Pnorm}
    There is a constant $C$ big enough, such that $\left|P(f,f)-P(g,g)\right|<C\left|f-g\right|$ where $\left|\cdot\right|$ denotes a proper metric. 
\end{assumption}
Part of the proof for the metric $\mathrm{d}_2$ defined in
    $P_s(\mathbb{R}^d)$ space (see \cite{TV_ProbMetrics}) can be found in the appendix.

Under this assumption, considering $P(M,M)=Q(M,M)+\mu M=\mu M$, one has
\begin{equation}
    \left|P(f,f)-\mu M\right|<C\left|f-M\right|.
\end{equation}
The derivation and the proof for both approaches being AP will be presented below.
We first show that ExpRK-F is AP for any given explicit Runge-Kutta scheme.

For AP property, one needs to show that as $\varepsilon\to 0$, the scheme gives correct
Euler limit. To do this, basically one needs to prove that $f$ goes to the
Maxwellian function whose macroscopic quantities solve the Euler equation
(\ref{EulerLimit}). 

Let us define
\begin{align}
    d_i&=\left|f^{(i)}-\tilde{M}\right|,\hspace{0.5cm}
    D_i=\left|v\cdot\nabla_xf^{(i)}\right|,\hspace{0.5cm}
    d_0=\left|f^{n}-\tilde{M}\right|,
\hspace{0.5cm}\vec{e}=[1,1,\cdots,1]^T,\nonumber\\
\hspace{0.5cm}\vec{d}&=\left[d_1,d_2,\cdots,d_\nu\right],\hspace{0.5cm}\vec{D}=\left[D_1,D_2,\cdots,D_\nu\right]^T.\label{def_forAPProof}
\end{align}
Moreover $\mathbb{A}$ is a lower-triangular matrix and $\mathbb{E}$ is a diagonal matrix given by
$$\mathbb{A}_{ij}=\frac{\lambda}{\mu}a_{ij}e^{(c_j-c_i)\lambda},\hspace{0.5cm}\mathbb{E}=\text{diag}\{e^{-c_1\lambda},e^{-c_2\lambda},\cdots,e^{-c_\nu\lambda}\}.$$
\begin{lemma}
    Based on the definitions above, for ExpRK-F one has
    \begin{equation*}
\vec{d}\leq d_0\left(\mathbb{I}-C\mathbb{A}\right)^{-1}\cdot\mathbb{E}\cdot\vec{e}+\varepsilon\left(\mathbb{I}-C\mathbb{A}\right)^{-1}\cdot\mathbb{A}\cdot\vec{{D}}
    \end{equation*}
\end{lemma}
\begin{proof}
    It is just direct derivation from (\ref{scheme_ExpRK})
\begin{align}
    (f^{(i)}-\tilde{M})e^{c_i\lambda}&=(f^{n}-\tilde{M})+\sum_{j=1}^{i-1}a_{ij}\frac{\lambda}{\mu}e^{c_{j}\lambda}(P^{(j)}-\mu \tilde{M}-\varepsilon v\cdot\nabla_xf^{(j)})\label{scheme_1}
\end{align}
By taking the norm, adopting the triangle inequality, and make use of the 
assumption that $\left|P(f)-\mu\tilde{M}\right|\leq C\left|f-\tilde{M}\right|$, one gets
\begin{align}
    \left|f^{(i)}-\tilde{M}\right|&\leq\left|f^n-\tilde{M}\right|e^{-c_i\lambda}+\sum_ja_{ij}\frac{\lambda}{\mu}e^{(c_j-c_i)\lambda}\left(C\left|f^{(j)}-\tilde{M}\right|+\varepsilon \left|v\cdot\nabla_xf^{(j)}\right|\right)\label{scheme_2}
\end{align}
Written in the matrix form, it becomes
\begin{align}
    \left(\begin{array}{c}d_1\\d_2\\ \vdots\\d_\nu\end{array}\right)&\le\mathbb{E}\left(\begin{array}{c}d_0\\d_0\\ \vdots\\d_0\end{array}\right)+C\mathbb{A}\left(\begin{array}{c}d_1\\d_2\\ \vdots\\d_\nu\end{array}\right)+\varepsilon\mathbb{A}\left(\begin{array}{c}{D}_1\\{D}_2\\\vdots\\{D}_\nu\end{array}\right)\nonumber
\end{align}
Thus
\begin{align}
    \vec{d}&\leq d_0\mathbb{E}\cdot\vec{e}+C\mathbb{A}\cdot\vec{d}+\varepsilon\mathbb{A}\cdot\vec{{D}}\label{scheme_3}\\
\vec{d}&\leq d_0\left(\mathbb{I}-C\mathbb{A}\right)^{-1}\cdot\mathbb{E}\cdot\vec{e}+\varepsilon\left(\mathbb{I}-C\mathbb{A}\right)^{-1}\cdot\mathbb{A}\cdot\vec{{D}}\label{scheme_4}
\end{align}
which completes the proof.
\end{proof}
\begin{lemma}
Define
\begin{align}\label{estimate_distancefn_condense_R}
    R_1(\lambda)&=e^{-\lambda}\left(1+\frac{C\lambda}{\mu}\vec{b}\cdot\mathbb{E}^{-1}\left(\mathbb{I}-C\mathbb{A}\right)^{-1}\mathbb{E}\cdot\vec{e}\right)\\
    \vec{R}_2(\lambda)&=\frac{\varepsilon\lambda}{\mu}e^{-\lambda}\vec{b}\cdot\mathbb{E}^{-1}\cdot(\mathbb{I}-C\mathbb{A})^{-1}\cdot\left(\mathbb{I}+C\mathbb{A}\right)
\end{align}
then for scheme (\ref{scheme_ExpRK}) we have
\begin{equation}\label{estimate_distancefn_condense}
    \left|f^{n+1}-\tilde{M}\right|\leq\left|f^{n}-\tilde{M}\right|R_1(\lambda)+\vec{R}_2\cdot\vec{D}
\end{equation}
\end{lemma}
\begin{proof}
    It is just a simple derivation. Define 
\begin{align}\label{scheme_distancek}
    k_i=\frac{h}{\varepsilon}(P^{(i)}-\mu \tilde{M}-\varepsilon v\cdot\nabla_xf^{(i)})e^{c_i\lambda}.
\end{align}
Evidently, the previous lemma leads to 
\begin{align}
\vec{|k|}\le\frac{\lambda}{\mu}\mathbb{E}^{-1}\cdot\left(C\vec{d}+\varepsilon\vec{D}\right).
\end{align}
Back to (\ref{scheme_ExpRK}), one has 
\begin{equation}
        \left(f^{n+1}-\tilde{M}\right)=\left(f^{n}-\tilde{M}\right)e^{-\lambda}+\sum_{s=1}^{\nu}b_ik_{i}e^{-\lambda},
        \end{equation}        
which implies        
\begin{subequations}
    \begin{align}\label{scheme_distancefn}
        \left|f^{n+1}-\tilde{M}\right|\leq& d_0e^{-\lambda}+\frac{\lambda}{\mu}e^{-\lambda}\vec{b}^T\cdot\mathbb{E}^{-1}\cdot\left(C\vec{d}+\varepsilon\vec{D}\right)\\
                                                             \leq& e^{-\lambda}\left(d_0+\frac{\lambda}{\mu}\vec{b}\cdot\mathbb{E}^{-1}\cdot\left(C\left(\mathbb{I}-C\mathbb{A}\right)^{-1}\cdot\left(d_0\mathbb{E}\cdot\vec{e}+\varepsilon\mathbb{A}\cdot\vec{D}\right)+\varepsilon\vec{D}\right)\right)\\
                                                             \leq& d_0e^{-\lambda}\left(1+\frac{C\lambda}{\mu}\vec{b}\cdot\mathbb{E}^{-1}\cdot(\mathbb{I}-C\mathbb{A})^{-1}\cdot\mathbb{E}\cdot\vec{e}\right)\\
                                                                 &+\frac{\varepsilon\lambda}{\mu}e^{-\lambda}\vec{b}\cdot\mathbb{E}^{-1}\cdot\left(\mathbb{I}-C\mathbb{A}\right)^{-1}\cdot\left(\mathbb{I}+C\mathbb{A}\right)\cdot\vec{D}.
\end{align}
\end{subequations}
Here $\vec{b}=\left[b_1,b_2,\cdots,b_\nu\right]$ is a row vector. The result
(\ref{scheme_4}) is also used.
Plug in the definition of $R_1$ and $R_2$, one gets 
\begin{equation}\label{estimate_distancefn_condense}
    \left|f^{n+1}-\tilde{M}\right|\leq\left|f^{n}-\tilde{M}\right|R_1(\lambda)+\vec{R}_2(\lambda)\cdot\vec{D}.
\end{equation}
\end{proof}
The two lemmas above gives us the estimation of the convergence rate towards
the Maxwellian. The smaller $R_1$ is, the faster the function converges. $R_2$
represents the drift from the transportation, and is expected to be small in
the limit. Also, the matrix $\mathbb{A}$ is usually a
lower triangular matrix, and a strict lower triangular matrix for explicit
Runge-Kutta, thus it is a nilpotent.
\begin{theorem}
    The method ExpRK-F defined by (\ref{scheme_ExpRK}) is AP for general 
    explicit Runge-Kutta method with $0\leq c_1\leq c_2\leq\cdots \leq c_\nu<1$.
\end{theorem}
\begin{proof}
Obviously if $R_1(\lambda)=O(\varepsilon)$ and $R_2(\lambda)=O(\varepsilon)$ for
$\varepsilon$ small enough, the theorem holds. In fact, for explicit
Runge-Kutta method, $\mathbb{A}$ is a strict lower triangular matrix, and thus
a nilpotent, then one has the following 
\begin{subequations}
\begin{align}
    \mathbb{E}^{-1}\left(\mathbb{I}-C\mathbb{A}\right)^{-1}\mathbb{E}&=\mathbb{E}^{-1}\left(\mathbb{I}+C\mathbb{A}+C^2\mathbb{A}^2+\cdots+C^{\nu-1}\mathbb{A}^{\nu-1}\right)\mathbb{E}\\
                                                                    &=\mathbb{I}+\mathbb{B}+\mathbb{B}^2+\cdots+\mathbb{B}^{\nu-1}
\end{align}
\end{subequations}
where $\mathbb{A}^\nu=0$, definition $\mathbb{B}=C\mathbb{E}^{-1}\mathbb{A}\mathbb{E}$ and
$\mathbb{E}^{-1}\mathbb{A}^2\mathbb{E}=\mathbb{E}^{-1}\mathbb{A}\mathbb{E}\mathbb{E}^{-1}\mathbb{A}\mathbb{E}$
are used. According to the definition of $\mathbb{A}$ and $\mathbb{E}$, it can
be computed that 
\begin{align*}
    \mathbb{B}_{ij}=C\mathbb{A}_{ij}e^{c_i\lambda-c_j\lambda}=\frac{C\lambda}{\mu}a_{ij}.
\end{align*}
Thus $\mathbb{I}+\sum_k\mathbb{B}^k$ is a matrix such that: the element on the
$k$th diagonal is of order $O(\lambda^k)$. This leads to obvious result 
\begin{equation*}
    R_1(\lambda)=e^{-\lambda}\left(1+\frac{C\lambda}{\mu}\vec{b}\cdot\mathbb{E}^{-1}\left(\mathbb{I}-\mathbb{A}\right)^{-1}\mathbb{E}\cdot\vec{e}\right)=O(e^{-\lambda}\lambda^{\nu-1})<O(\varepsilon)
\end{equation*}
Similar analysis can be carried to $R_2(\lambda)$ to show that it vanishes to
zero as $\varepsilon\to 0$.\\
So as $\varepsilon\to 0$, $|f^{n+1}-\tilde{M|}\to 0$. By definition,
$\tilde{M}$ is defined by macroscopic quantities computed directly from the
limit Euler equation, thus the numerical scheme is AP, which finishes
the proof.
\end{proof}
The derivation of the scheme ExpRK-V is essentially the same, and in the end, 
one still has, in a condense form 
\begin{equation}\label{estimate_distancefn_condense}
\left|f^{n+1}-{M}^{n+1}\right|\leq\left|f^{n}-{M}^n\right|R_1(\lambda)+\vec{R}_2\cdot\vec{D}
\end{equation}
with $R_1$, $\vec{R}_2$, $\mathbb{E}$, $\mathbb{A}$ defined in the same way as
in (\ref{def_forAPProof}), but $D_i=|v\cdot\nabla_xf^{(i)}+\partial_tM^{(i)}|$.
Following the same computations, one could prove that this method is AP too, but
the proof is omitted for brevity.
\begin{theorem}
    The method ExpRK-V defined by (\ref{scheme_M2}) is AP for general explicit Runge-Kutta method.
\end{theorem}
%

\section{Numerical Example}
\subsection{Convergence Rate Test}
In this example, we use smooth data to check the convergence rate of both 
methods. The problem is adopted from \cite{Filbet}: $1$ dimensional in 
$x$ and $2$ dimensional in $v$. Initial distribution is given by 
\begin{equation}\label{num_ini}
    f(t=0,x,v)=\frac{\rho_0(x)}{2}\left(e^{-\frac{|v-u_1(x)|^2}{T_0(x)}}+e^{\frac{|v-u_2(x)|^2}{T_0(x)}}\right)
\end{equation}
with
\begin{align*}
    \rho_0(x)&=\frac{1}{2}\left(2+\sin{\left(2\pi x\right)}\right),\\
       u_1(x)&=\left[0.75,-0.75\right]^T,\hspace{0.5cm}u_2(x)=\left[-0.75,0.75\right]^T,\\
       T_0(x)&=\frac{1}{20}\left(5+2\cos{\left(2\pi x\right)}\right).
\end{align*}
Domain is chosen as $x\in\left[0,1\right]$ and periodic boundary condition on $x$ is 
used.
    Note that the definition of $\rho_0$, $u_{1/2}$ and $T_0$ do not represent the 
    number density, average velocity and temperature.
    
As one can see, the initial data is summation of two Gaussian functions 
centered at $u_1$ and $u_2$ respectively, and is far away from the Maxwellian. 
To check the convergence rate, we use $N_x=128, 256, 512, 1024$ grid points on 
$x$ space, and $N_v=32$ points on $v$ space. Time stepping $\Delta t$ is 
chosen to satisfy CFL condition with CFL number being $0.5$. We measure the $L_1$ error 
of $\rho$ and compute the decay rate through the following formula \cite{YanJin_StrongAP} 
\begin{equation}
    \text{error}_{\Delta x}=\max_{t=t^n}{\frac{\|\rho_{\Delta x}(t)-\rho_{2\Delta 
    x}(t)\|_1}{\|\rho_{2\Delta x}(t)\|_1}},
\end{equation}
with $\Delta x=\frac{1}{N_x}$. Theoretically, a $k$th order numerical scheme 
should give $\text{error}_{\Delta x}<C\left(\Delta x\right)^k$ for $\Delta x$ 
small enough.

We compute this problem using spectral method \cite{MP_FastSpectralCollision} 
in $v$, WENO of order 3/5 \cite{Shu_WENOReview} for $x$. For time discretization, we use the second and third order Runge-Kutta from \cite{HNW}, Table 1 page 135. We denote the four schemes under consideration as ExpRK2-F, ExpRK2-V, ExpRK3-F and ExpRK3-V.

We compute the problem using the Maxwellian, and a distribution function 
away from the Maxwellian given above as initial data, for 
$\epsilon=1, 0.1, 10^{-3}, 10^{-6}$. Results are shown in Figure 
\ref{fig_Ex1_M1-2}. We also give the convergence rate Table \ref{tab_Ex1}. One 
can see that in kinetic regime, when $\varepsilon=1$, the two methods are 
almost the same, but as $\varepsilon$ becomes smaller, in the intermediate 
regime, for example $\varepsilon=0.1$ for the second order schemes and $\varepsilon=10^{-3}$ for second and third order schemes with Maxwellian data, ExpRK-V performs better then ExpRK-F. In the hydrodynamic regime, however, 
the two methods give similar results again shown by the two pictures for 
$\varepsilon=10^{-6}$. It is remarkable that the third order methods achieve almost order $5$ (the maximum achievable by the WENO solver) in many regimes. 

\begin{figure}
    \begin{subfigure}
                \centering
                \includegraphics[width=0.5\textwidth,height=0.2\textheight]{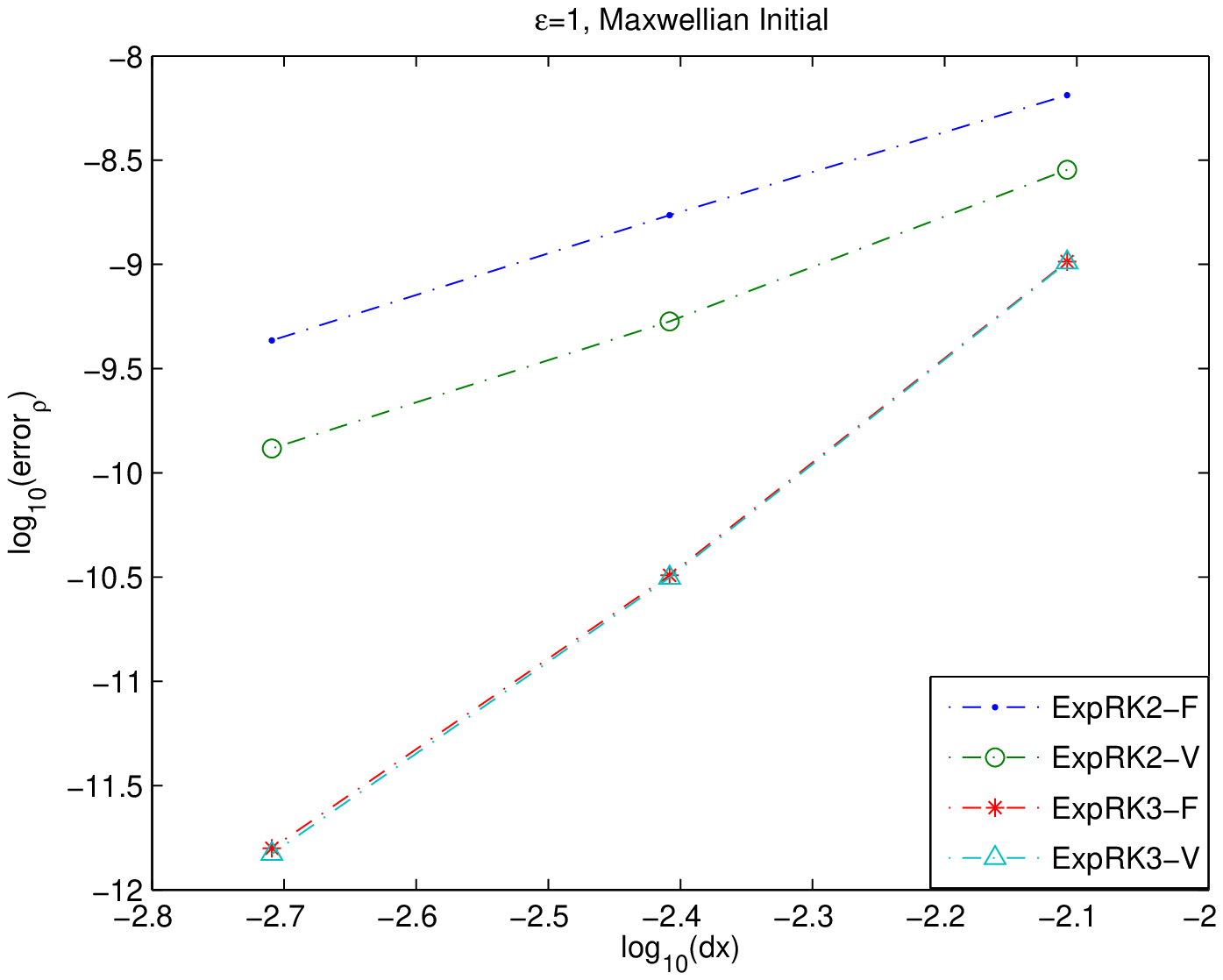}
        \end{subfigure}%
        \begin{subfigure}
                \centering
            \includegraphics[width=0.5\textwidth,height=0.2\textheight]{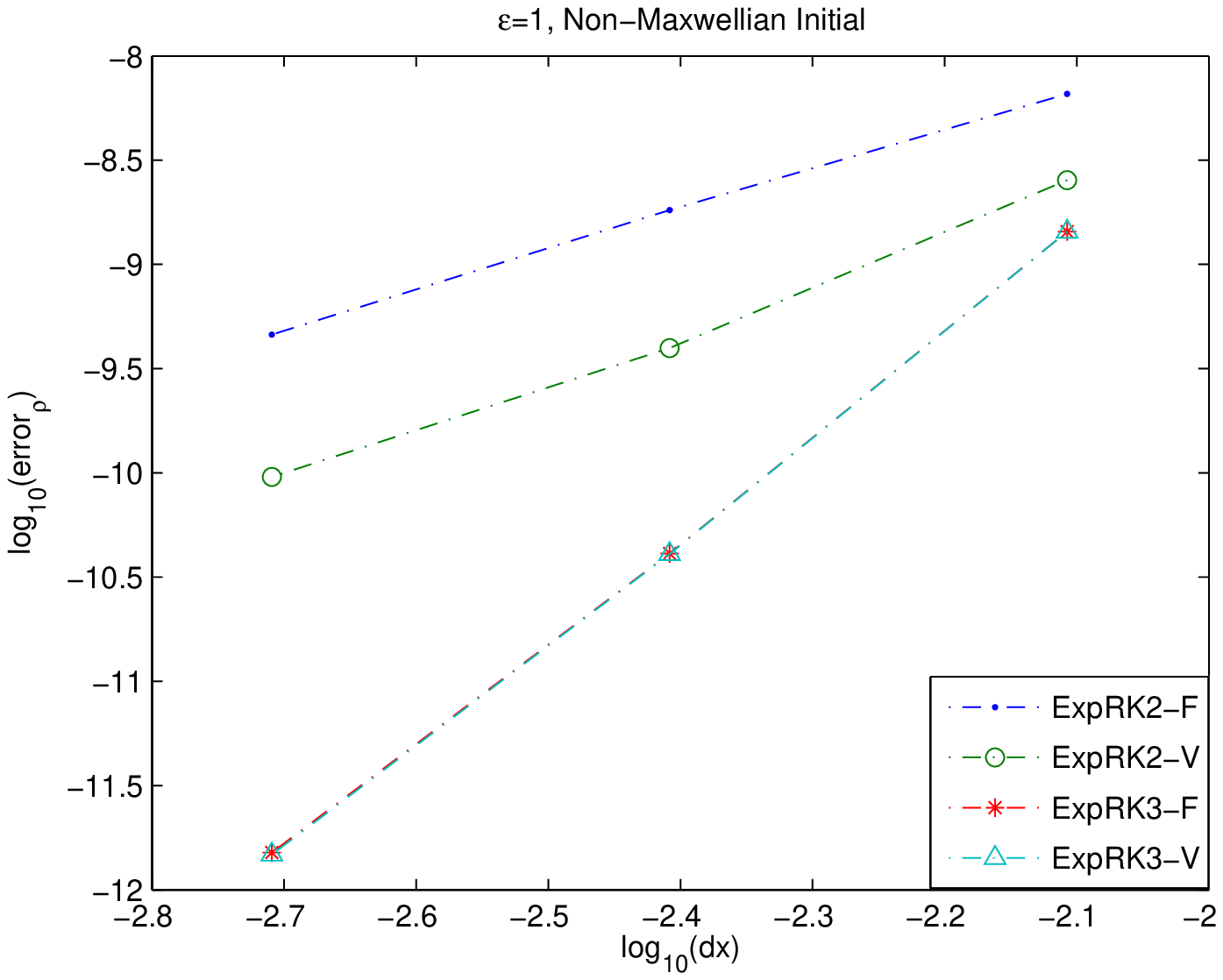}
        \end{subfigure}
    \begin{subfigure}
                \centering
                \includegraphics[width=0.5\textwidth,height=0.2\textheight]{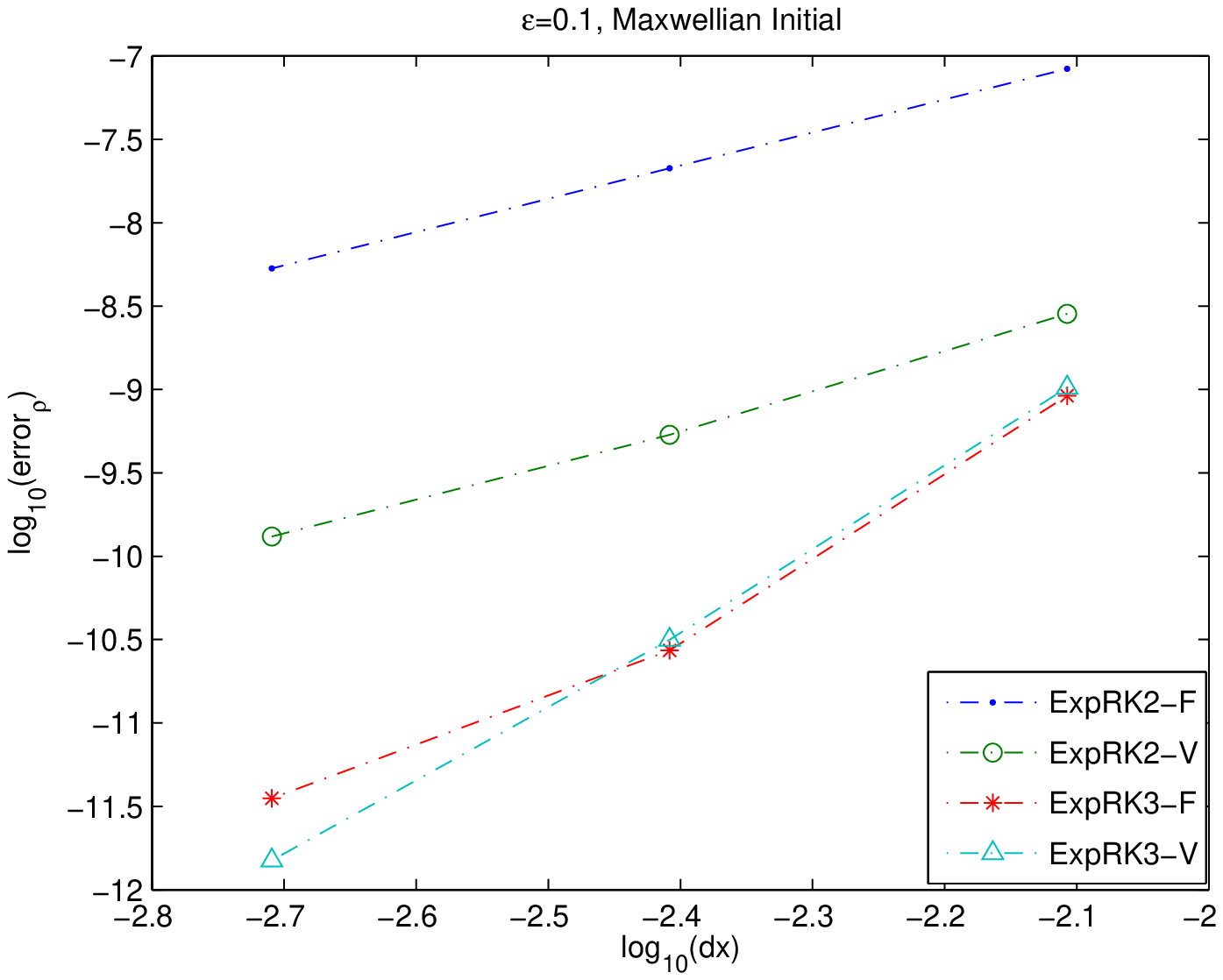}
        \end{subfigure}%
        \begin{subfigure}
                \centering
                \includegraphics[width=0.5\textwidth,height=0.2\textheight]{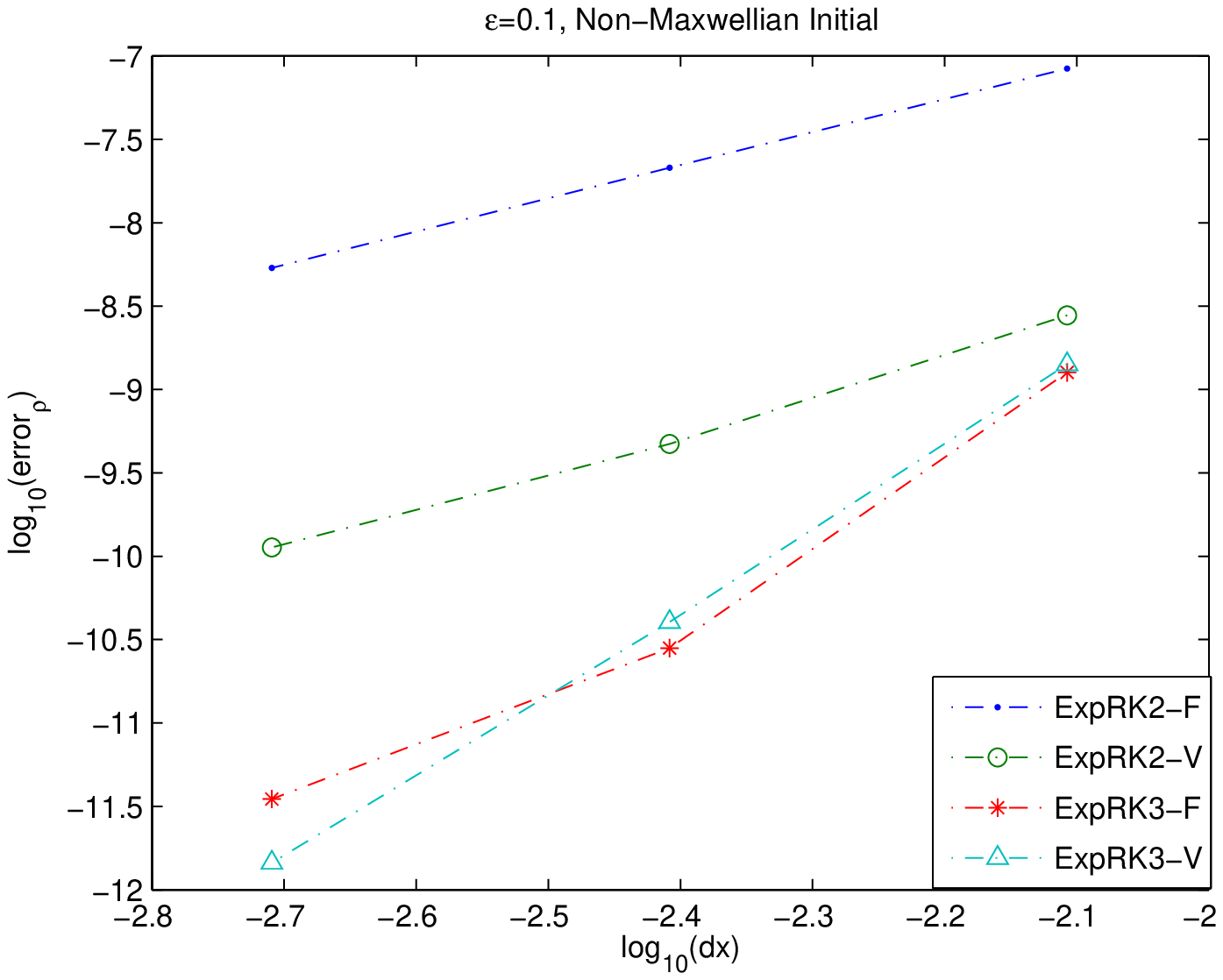}
        \end{subfigure}
    \begin{subfigure}
                \centering
                \includegraphics[width=0.5\textwidth,height=0.2\textheight]{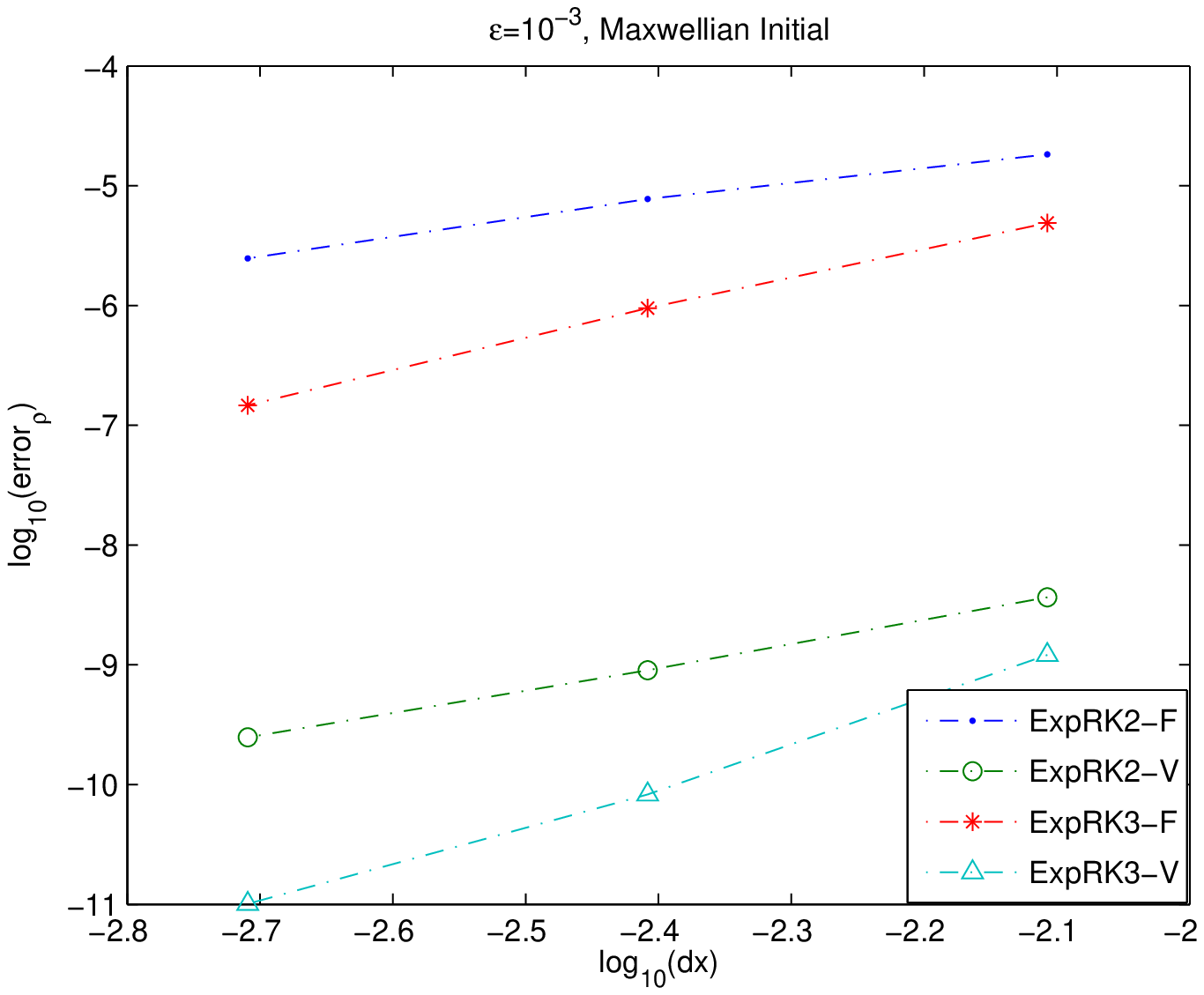}
        \end{subfigure}%
        \begin{subfigure}
                \centering
                \includegraphics[width=0.5\textwidth,height=0.2\textheight]{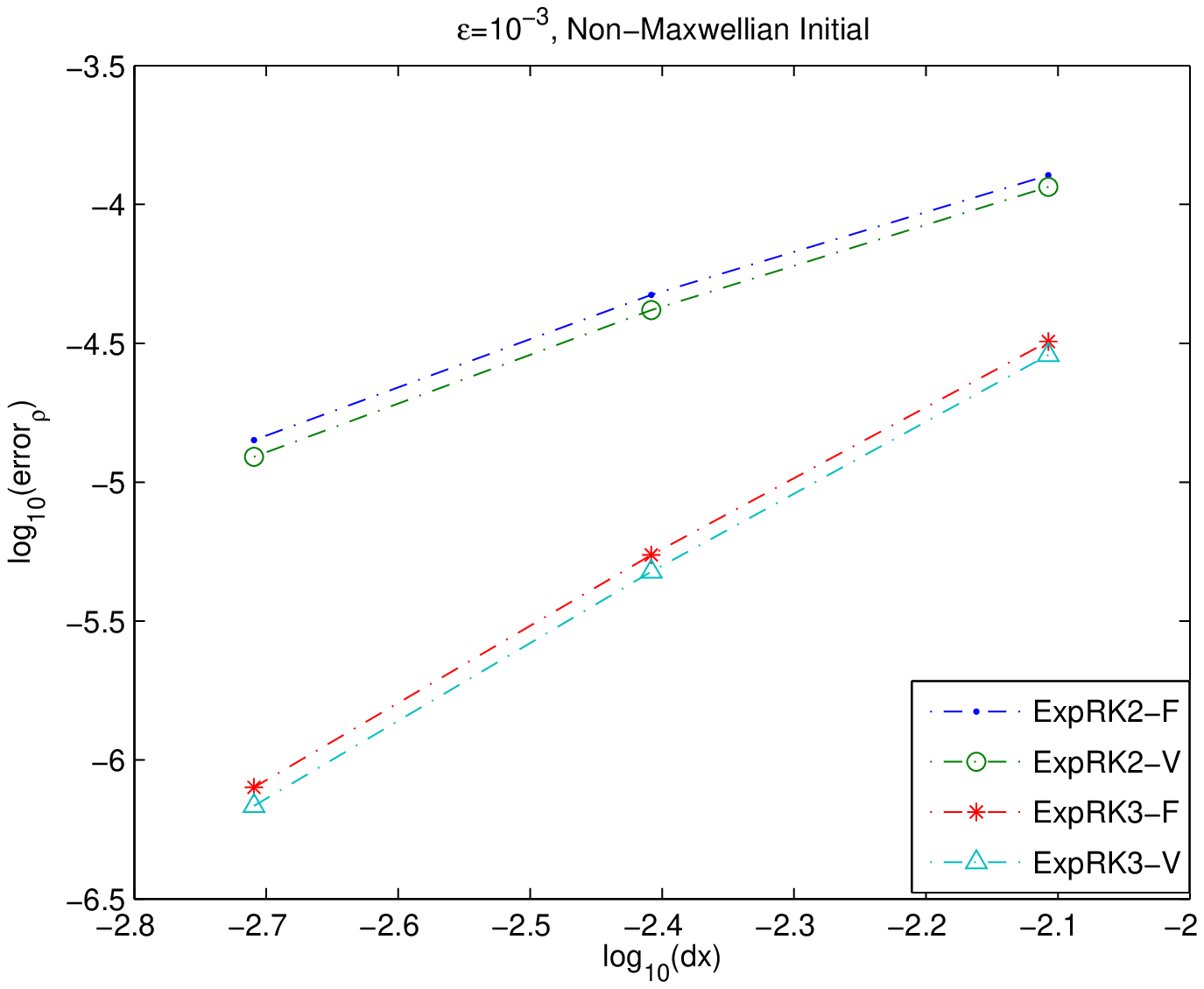}
        \end{subfigure}        \label{fig_Ex1_M1_ep6}
    \begin{subfigure}
                \centering
                \includegraphics[width=0.5\textwidth,height=0.2\textheight]{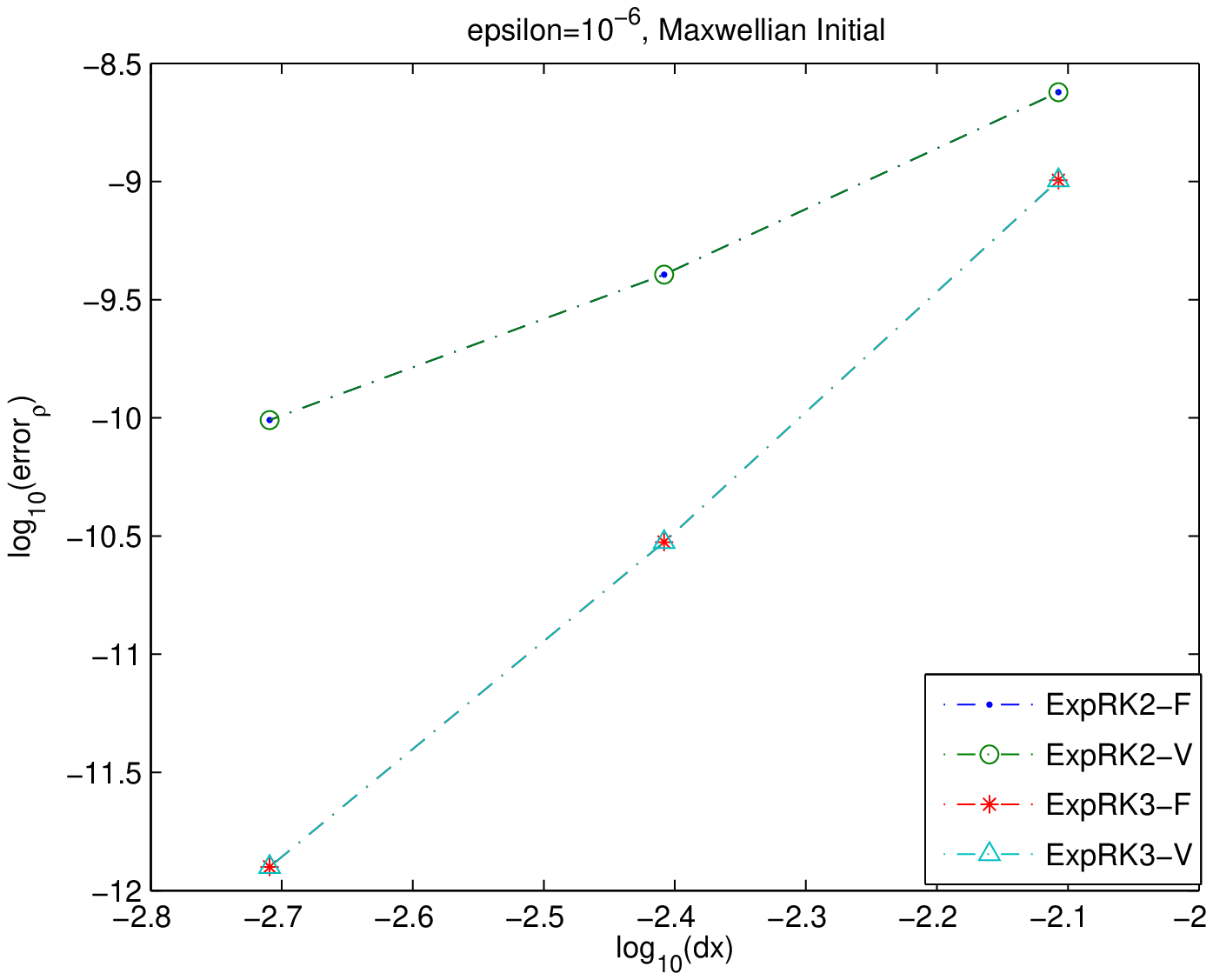}
        \end{subfigure}%
        \begin{subfigure}
                \centering
                \includegraphics[width=0.5\textwidth,height=0.2\textheight]{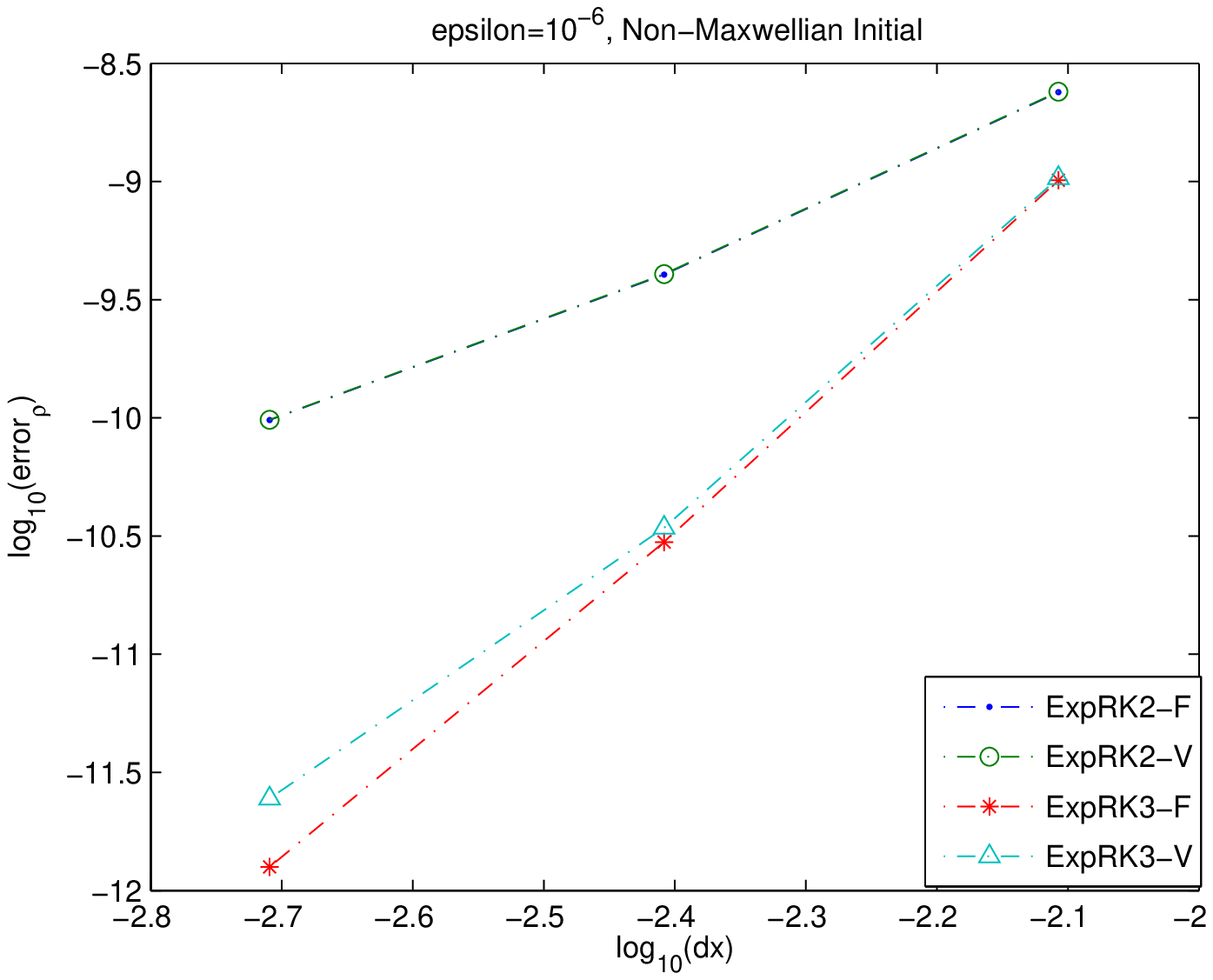}
        \end{subfigure}

        \label{fig_Ex1_M1-2}
        \caption{Convergence rate test. In each picture, 4 lines are plotted: 
        the lines with dots, circles, stars and triangles on them are given by 
    results of ExpRK2-F, ExpRK2-V, ExpRK3-F and ExpRK3-V respectively. The 
left column is for Maxwellian initial data, and the right column is for 
initial data away from Maxwellian (\ref{num_ini}). Each row, from the top to 
the bottom, shows results of $\epsilon=1/0.1/10^{-3}/10^{-6}$ respectively.}
    \end{figure}

\begin{table}[hbt]\label{tab_Ex1}
        \centering
        \small
        \begin{tabular}{|c|c|c|c|c|c|}
            \hline
            \multicolumn{2}{|c|}{Initial Distribution}&\multicolumn{2}{|c|}{Maxwellian Initial} & \multicolumn{2}{|c|}{Non-Maxwellian Initial}\\
            \hline
            \multicolumn{2}{|c|}{$N_x$}&$128-256-512$&$256-512-1024$&$128-256-512$&$256-512-1024$\\
            \hline
            \multirow{4}{*}{$\varepsilon=1$} & ExpRK2-F&1.91327&1.99502&1.84968&1.98504\\
                                          &ExpRK2-V&2.41608&2.02347&2.67733&2.05436\\
                                          &ExpRK3-F&4.99725&4.35014&5.12959&4.76788\\
                                          &ExpRK3-V&5.02508&4.40379&5.13515&4.79080\\
            \hline
            \multirow{4}{*}{$\varepsilon=0.1$}&ExpRK2-F&1.98218&1.99539&1.97725&1.99454\\
                                           &ExpRK2-V&2.41411&2.02293&2.56620&2.05830\\
                                           &ExpRK3-F&5.07621&2.94707&5.49587&3.00335\\
                                           &ExpRK3-V&5.02220&4.39651&5.13859&4.79264\\
            \hline
            \multirow{4}{*}{$\varepsilon=10^{-3}$}&ExpRK2-F&1.23711&1.64976&1.43331&1.73501\\
                                               &ExpRK2-V&2.02344&1.85924&1.47466&1.75496\\
                                               &ExpRK3-F&2.36140&2.69178&2.55225&2.78275\\
                                               &ExpRK3-V&3.86882&3.03223&2.59114&2.80353\\
            \hline
            \multirow{4}{*}{$\varepsilon=10^{-6}$}&ExpRK2-F&2.56137&2.04519&2.56137&2.04519\\
                                                  &ExpRK2-V&2.56137&2.04519&2.56383&2.04859\\
                                                  &ExpRK3-F&5.08829&4.56695&5.08830&4.56699\\
                                                  &ExpRK3-V&5.08830&4.56704&4.91909&3.80638\\
            \hline
        \end{tabular}
        \caption{Convergence rate for ExpRK methods with different initial 
        data, in different regimes.}
\end{table}

\subsection{A Sod Problem}
This simple example is adopted from \cite{YanJin_StrongAP} to check 
accuracy and AP of the numerical methods. It is a Riemann problem, and the 
solution to the associated Euler limit is a Sod problem.
\begin{equation*}
\begin{cases}
(\rho,u_x,u_y,T)=(1,0,0,1), & \text{if } x<0;\\
(\rho,u_x,u_y,T)=(1/8,0,0,1/4), & \text{if } x>0;
\end{cases}
\end{equation*}
In Figure \ref{fig_Ex2_bigep} (left), we show that when 
$\epsilon=0.01$ is comparably big, both the two new method proposed here match with the 
numerical results given by explicit scheme with dense mesh. Here the reference 
is given by Forward Euler with $\Delta x=1/500$ and $h=0.0001$. In 
Figure \ref{fig_Ex2_bigep} (right), AP property is shown: it is clear that for 
$\epsilon=10^{-6}$, numerical results capture the Euler limit -- the Euler 
limit is computed by kinetic scheme \cite{Perthame_KineticScheme}. All plots 
are given at time $t=0.2$.
\begin{figure}
    \begin{center}
    \includegraphics[height=5.5cm]{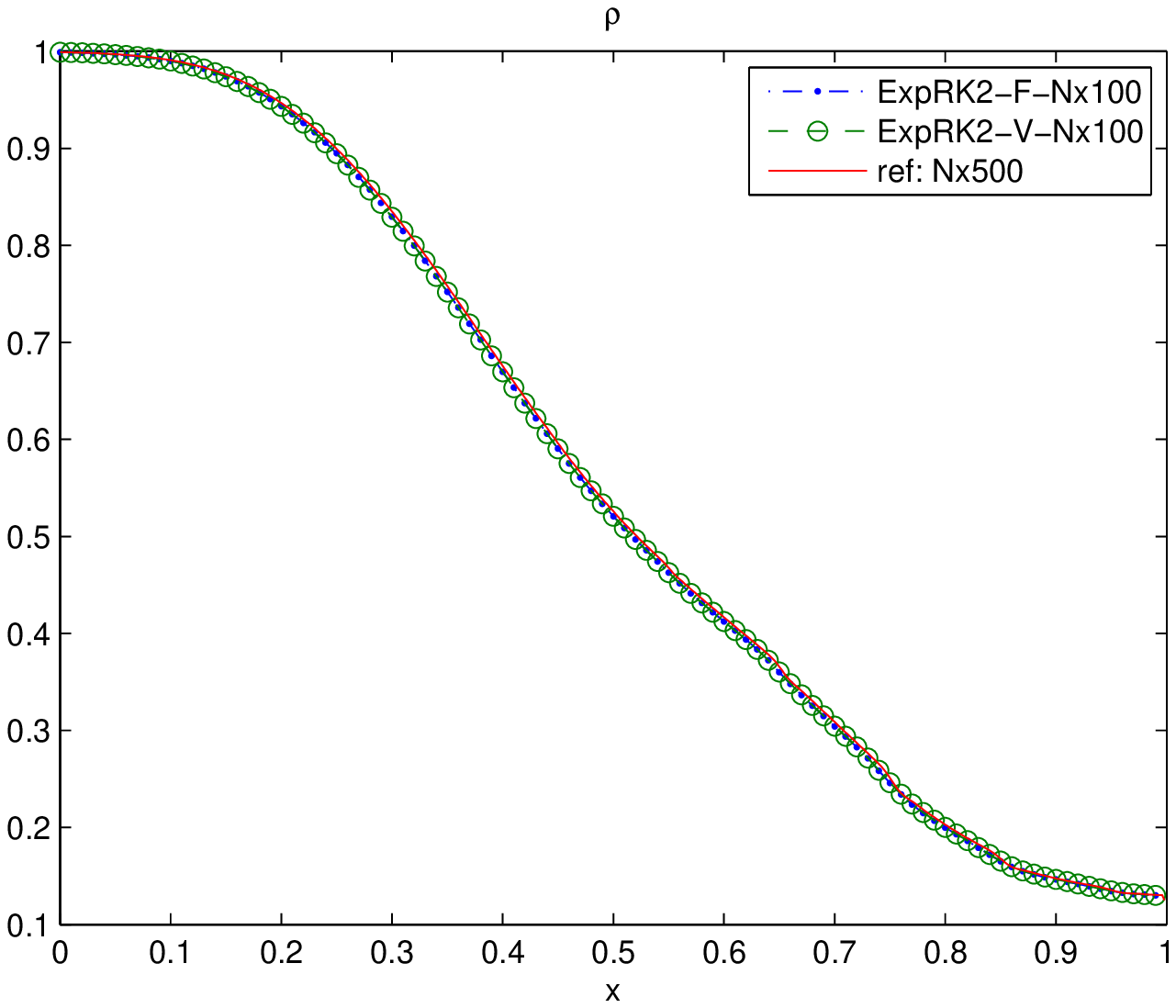}
     \includegraphics[height=5.5cm]{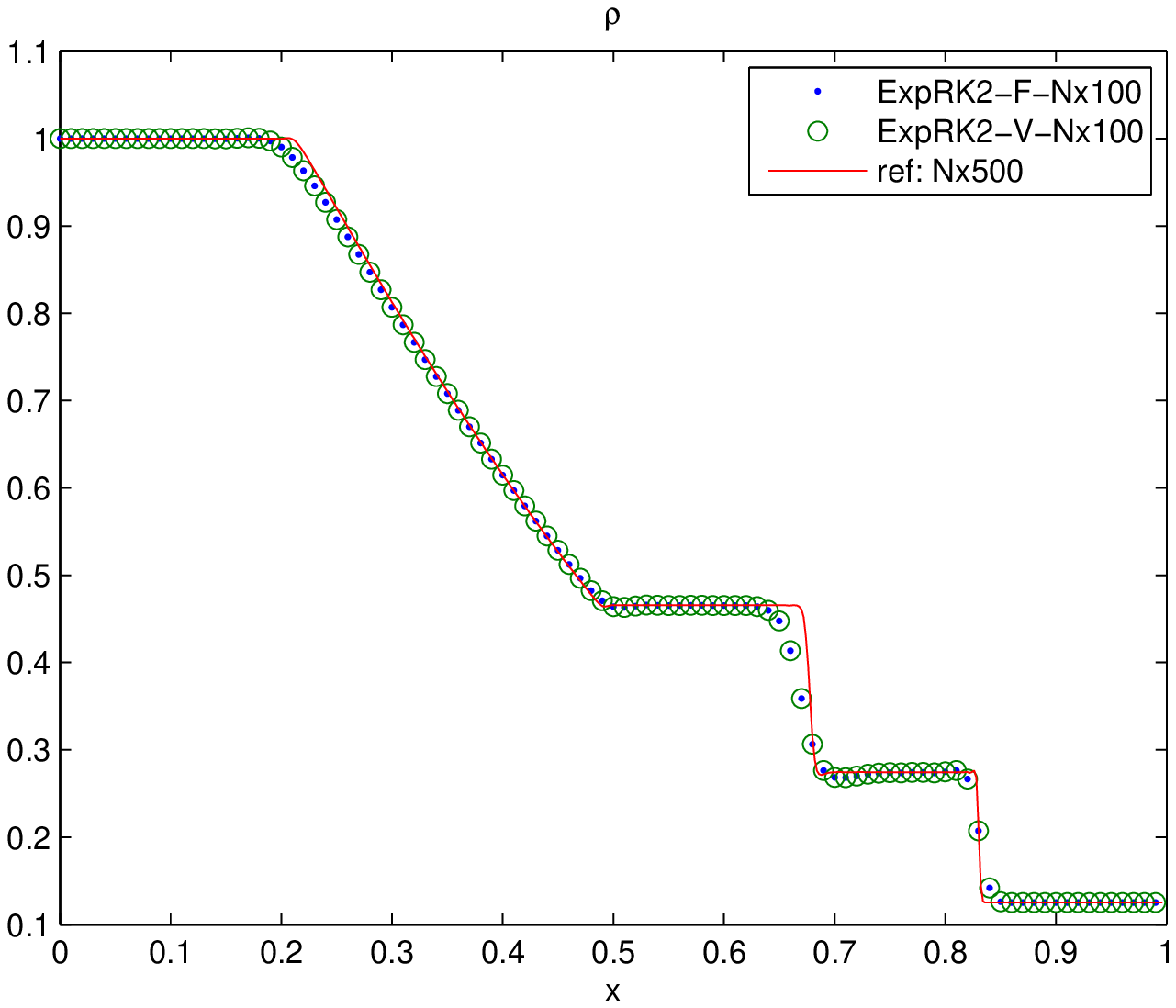}    
    \includegraphics[height=5.5cm]{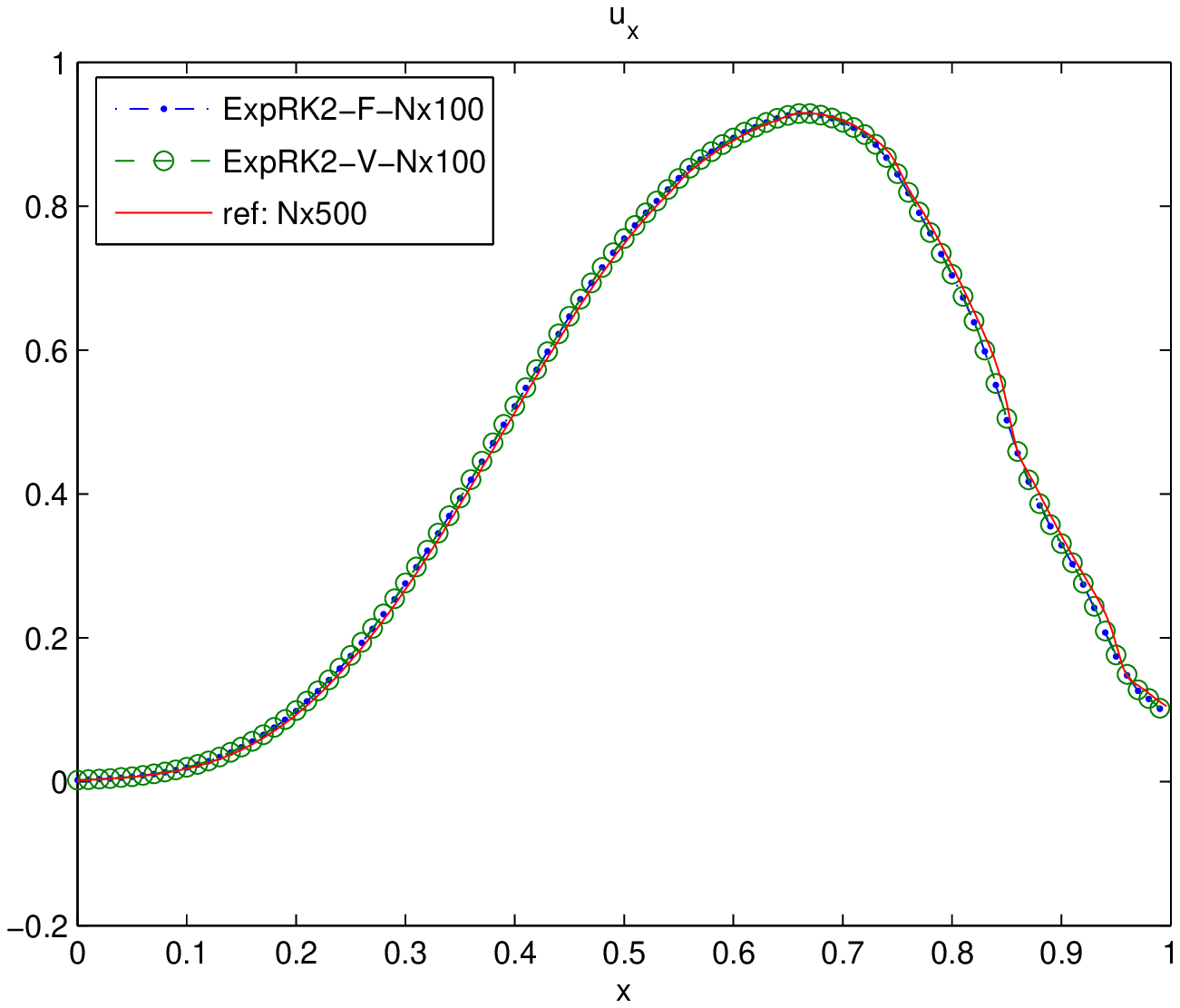}
    \includegraphics[height=5.5cm]{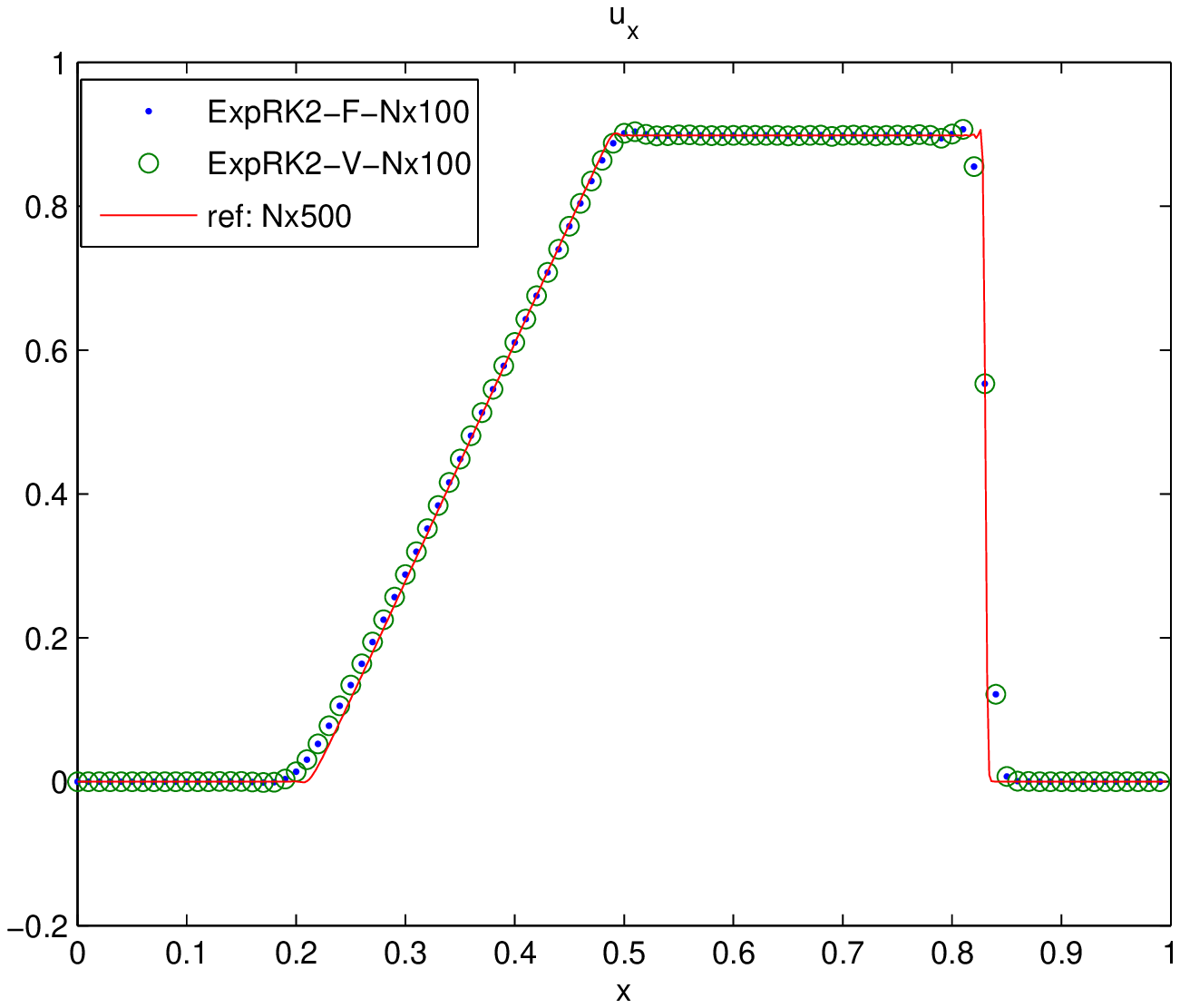}
    \includegraphics[height=5.5cm]{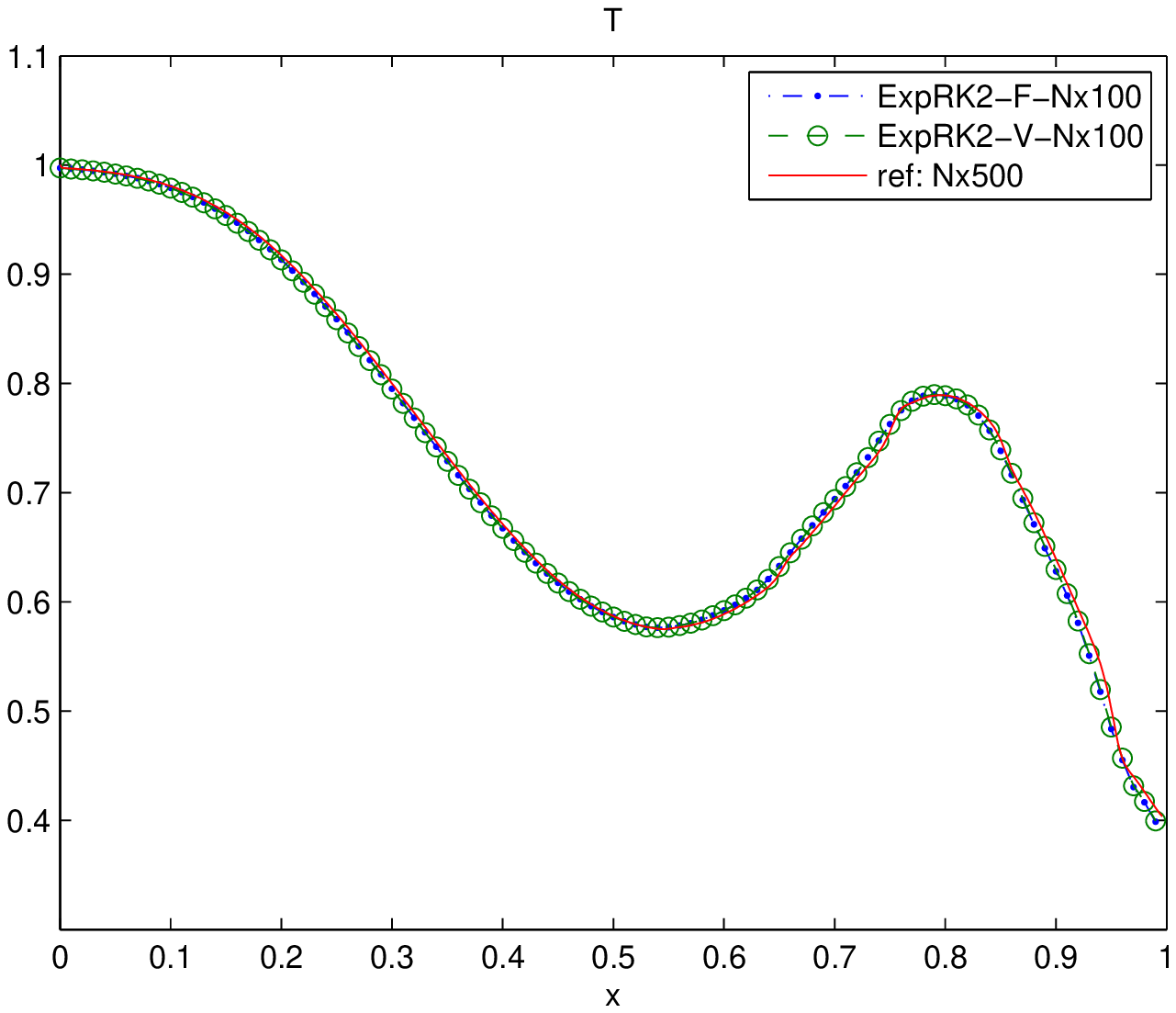}
    \includegraphics[height=5.5cm]{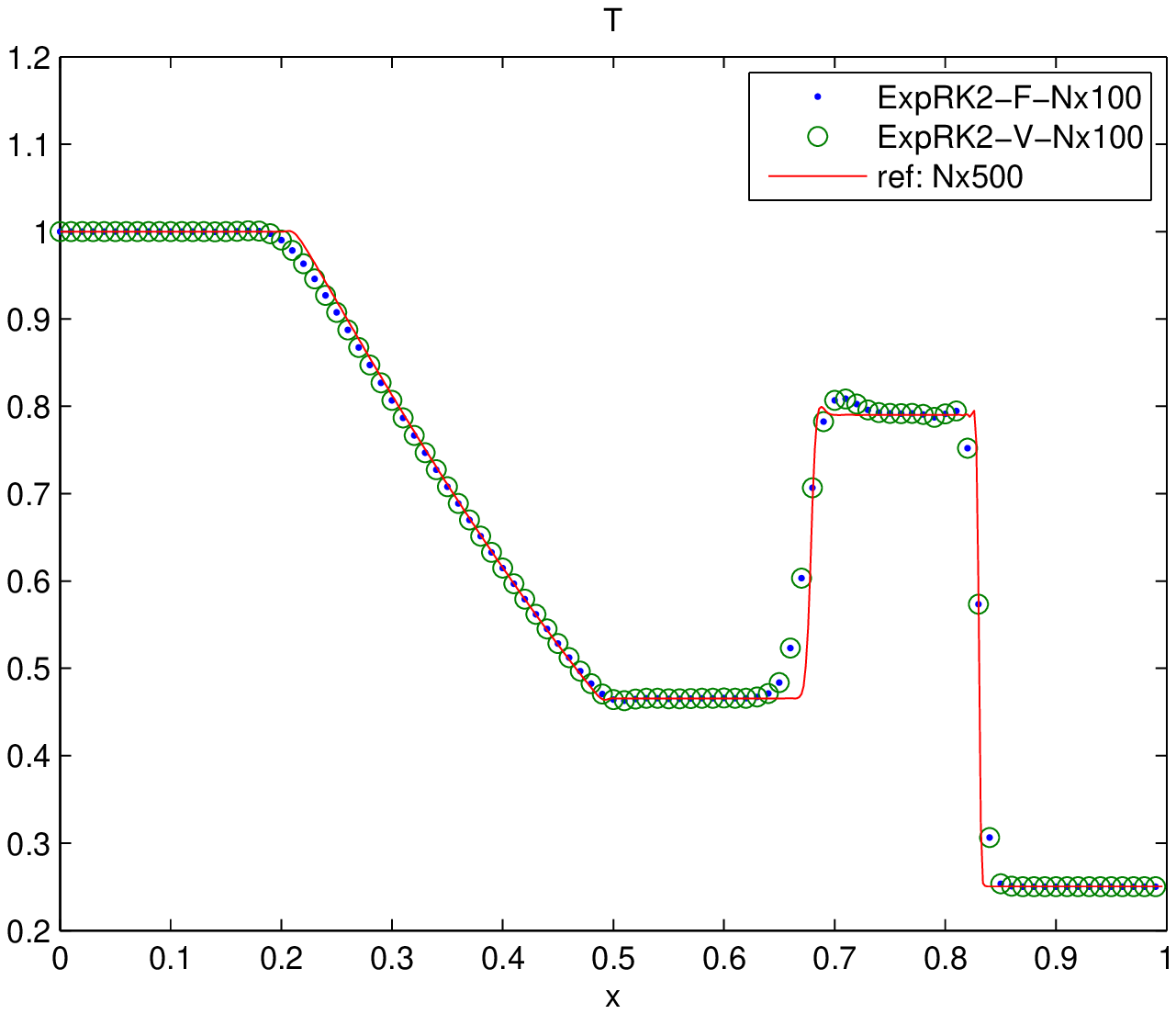}
    \caption{Consistency and AP. Left column: $\varepsilon=0.01$. The solid line is given by 
explicit scheme with dense mesh, while dots and circles are given by ExpRK2-F 
and ExpRK2-V respectively, both with $N_x=100$. $h=\Delta x/20$ satisfies the 
CFL condition with CFL number being $0.5$. Right column: For $\varepsilon=10^{-6}$, both methods capture the Euler 
    limit. The solid line is given by the kinetic scheme for the Euler equation, 
while the dots and circles are given by ExpRK2-F and ExpRK2-V. They perform 
well in rarefaction, contact line and shock.}
\end{center}
\label{fig_Ex2_bigep}
\end{figure}

\subsection{Mixing Regime}
In this example \cite{YanJin_StrongAP}, we show numerical results to a problem 
with mixing regime. This problem is difficult because $\varepsilon$ vary with 
respect to space. As what we do in the first example, we take identical data 
along one space direction, so it is $1D$ in space but $2D$ in velocity. An 
accurate AP scheme should be able to handle all $\varepsilon$ with 
considerably coarse mesh. Domain is chosen to be $x\in[-0.5,0.5]$, with 
$\varepsilon$ defined by 
\begin{equation}\label{def_epsilonMixingRegime}
    \epsilon=\begin{cases}
        \epsilon_0+0.5\left(\tanh{(6-20x)}+\tanh{(6+20x))}\right)\hspace{0.5cm}x<0.2;\\
        \epsilon_0\hspace{0.5cm}x>0.2
    \end{cases}
\end{equation}
where $\epsilon_0$ is $10^{-3}$. So $\varepsilon$ raise up from $10^{-3}$ to 
$O(1)$, and suddenly drop back to $10^{-3}$ as shown in Figure (\ref{fig_Ex3_epsilon}). Initial data is the give as 
\begin{equation}
    f(t=0,x,v)=\frac{\rho_0(x)}{4\pi T_0(x)}\left(e^{-\frac{|v-u_0(x)|^2}{2T_0(x)}}+e^{-\frac{|v+u_0(x)|^2}{2T_0(x)}}\right)
\end{equation}
with
\begin{equation}
    \begin{cases}
        \rho_0(x)=\frac{2+\sin{(2\pi x+\pi)}}{3},\\
        u_0(x)=\frac{1}{5}\left(\begin{array}{c}\cos{(2\pi x+\pi)}\\0\end{array}\right),\\
        T_0(x)=\frac{3+\cos{(2\pi x+\pi)}}{4}
    \end{cases}
\end{equation}
Periodic boundary condition on $x$ is applied.\\
We compute the problem using both methods proposed in this paper together with standard explicit Runge-Kutta 2 and 3 in time used as the underline methods in the exponential schemes. 

Results are plotted in 
Figure \ref{fig_Ex3_ep3}. The reference solution is computed with a very fine mesh in time. Both methods give excellent results simply taking a CFL condition of $0.5$ whereas explicit methods are forced to operate on a time scale 1000 times smaller. In particular, ExpRK3-V performs well uniformly on $\varepsilon$ by giving a more accurate description of the shock profiles.
\begin{figure}
    \centering
    \includegraphics[width=0.7\textwidth]{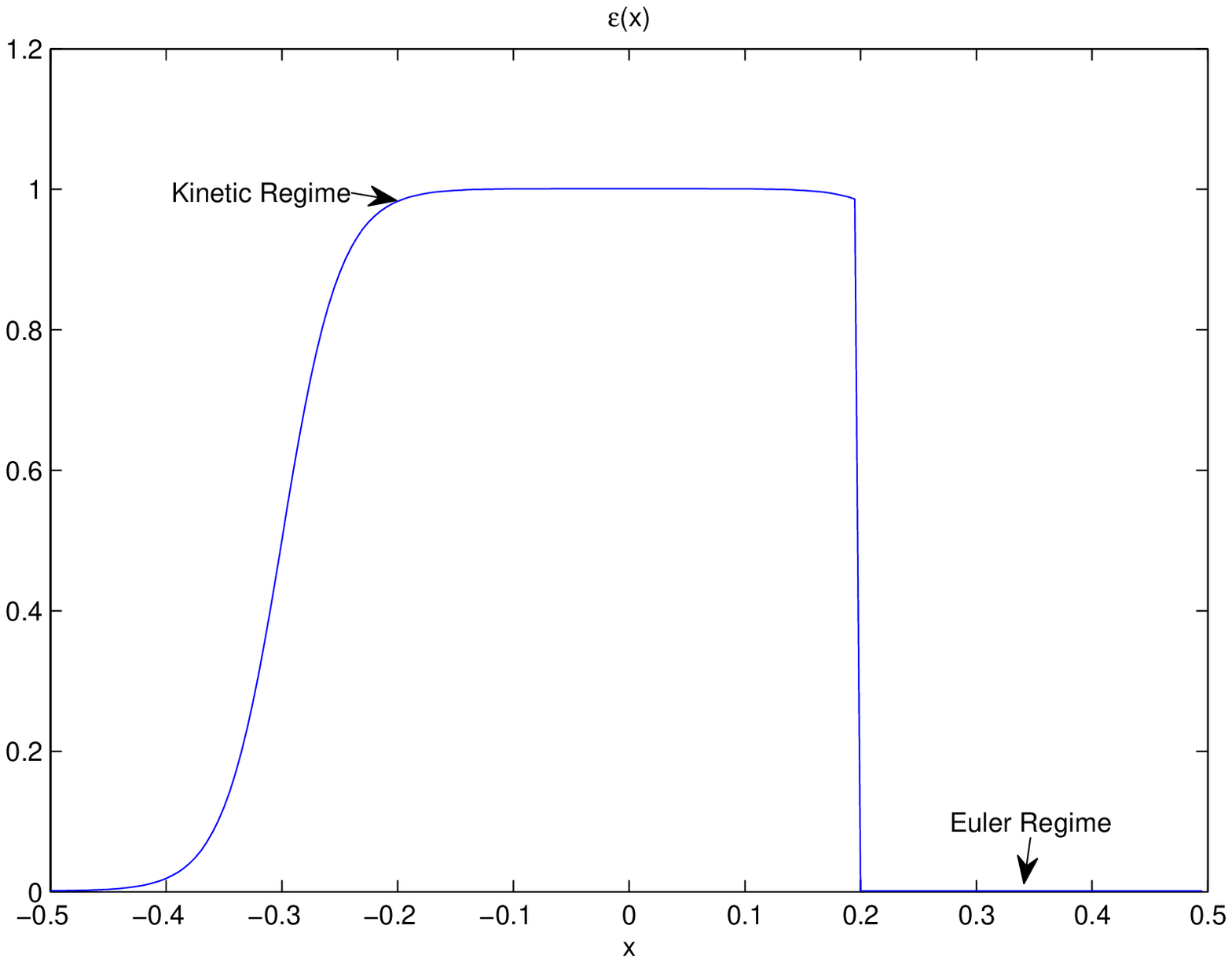}
    \caption{Mixing Regime: $\varepsilon(x)$}\label{fig_Ex3_epsilon}
\end{figure}
\begin{figure}
     \begin{subfigure}\centering
        \includegraphics[width=0.5\textwidth,height=0.25\textheight]{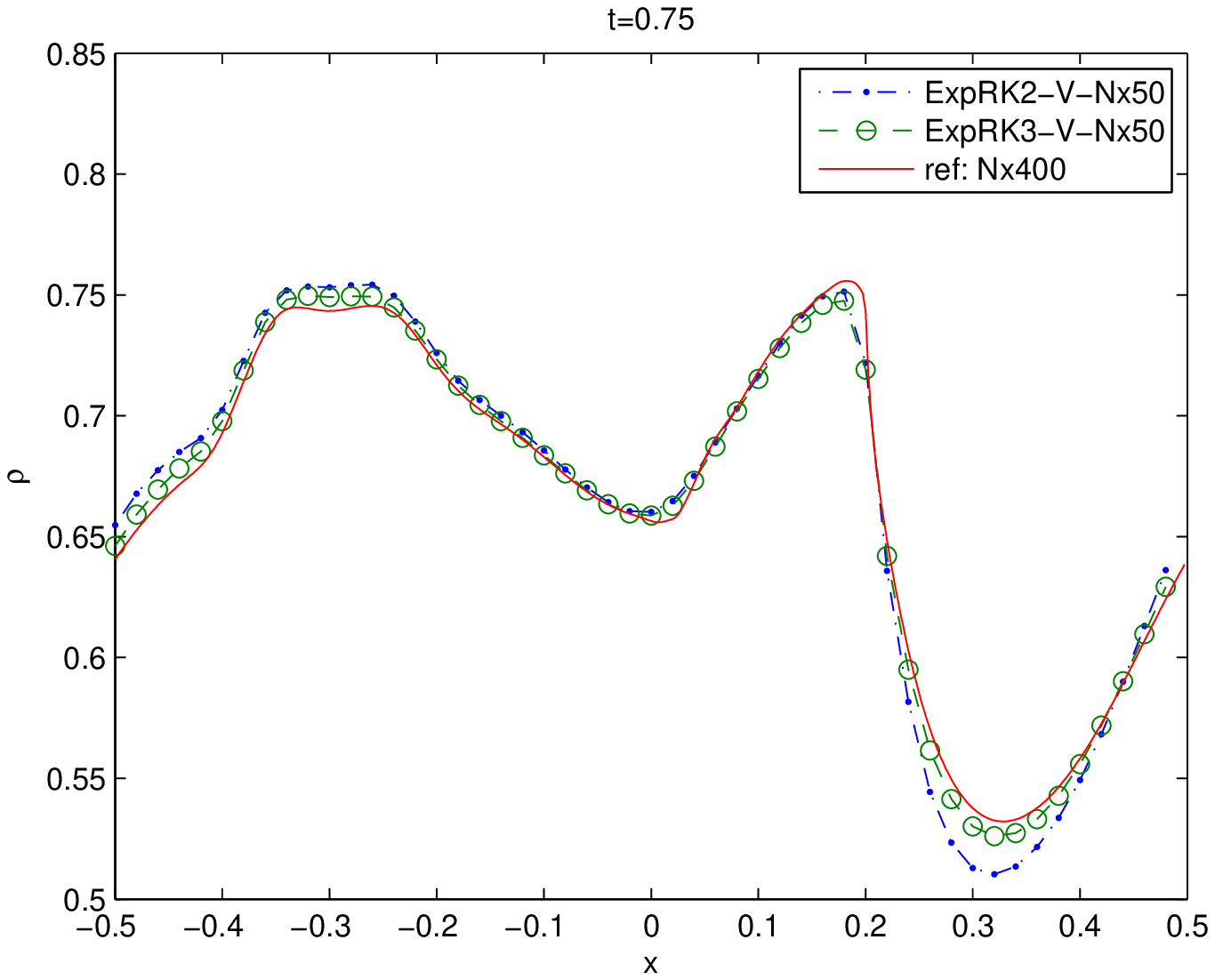}
     \end{subfigure}\begin{subfigure}\centering
        \includegraphics[width=0.5\textwidth,height=0.25\textheight]{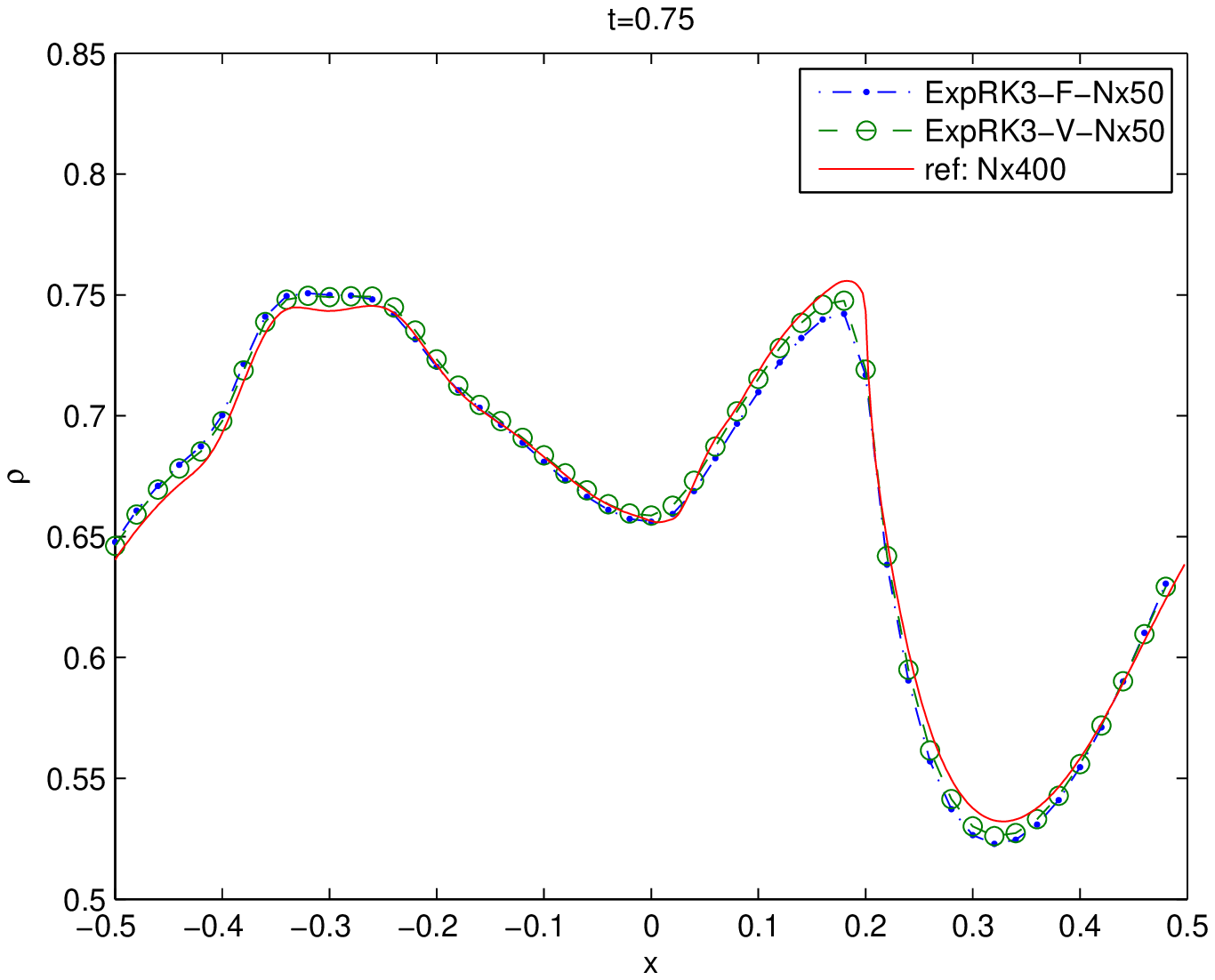}
    \end{subfigure}\begin{subfigure}\centering
       \includegraphics[width=0.5\textwidth,height=0.25\textheight]{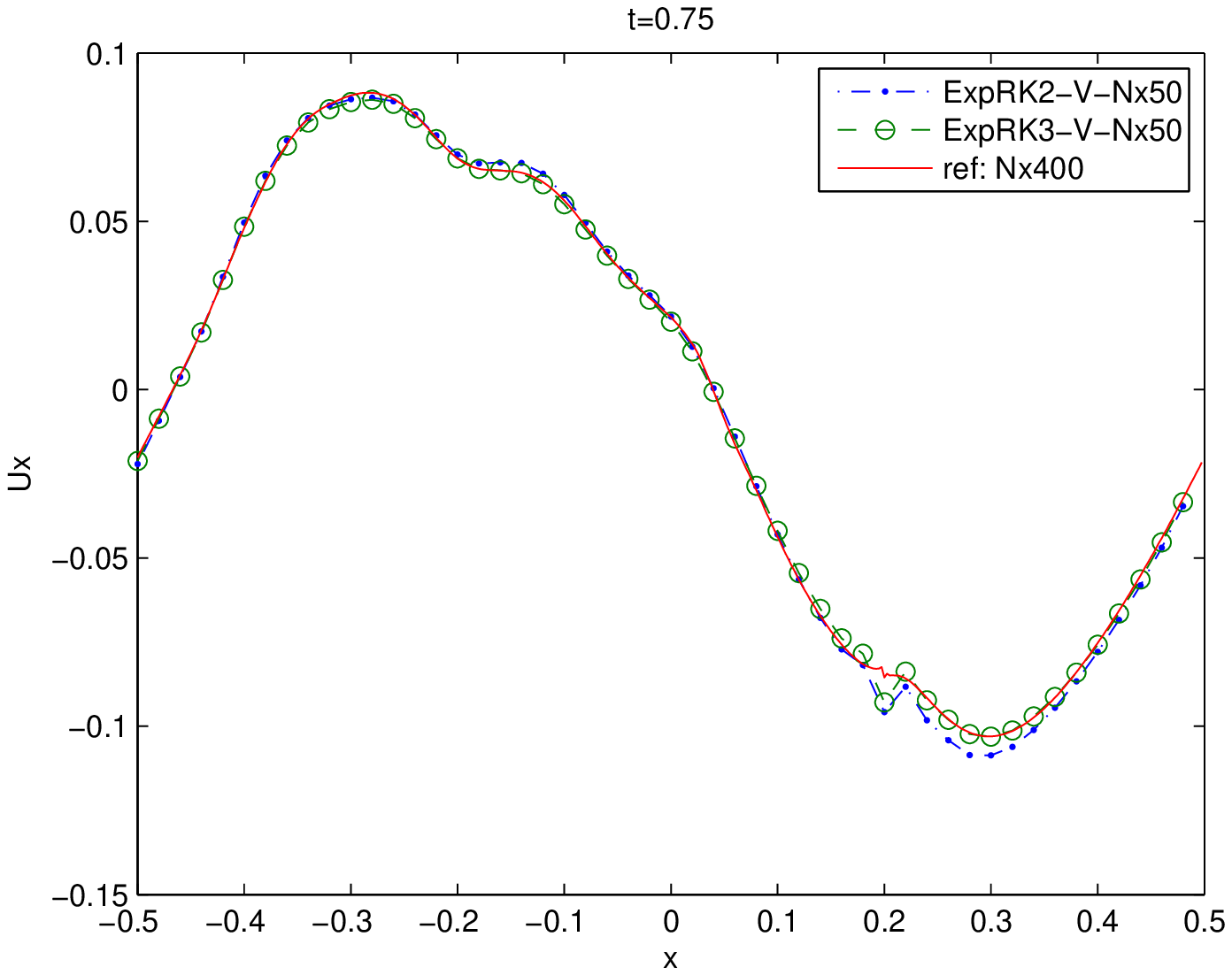}
    \end{subfigure}
     \begin{subfigure}\centering
        \includegraphics[width=0.5\textwidth,height=0.25\textheight]{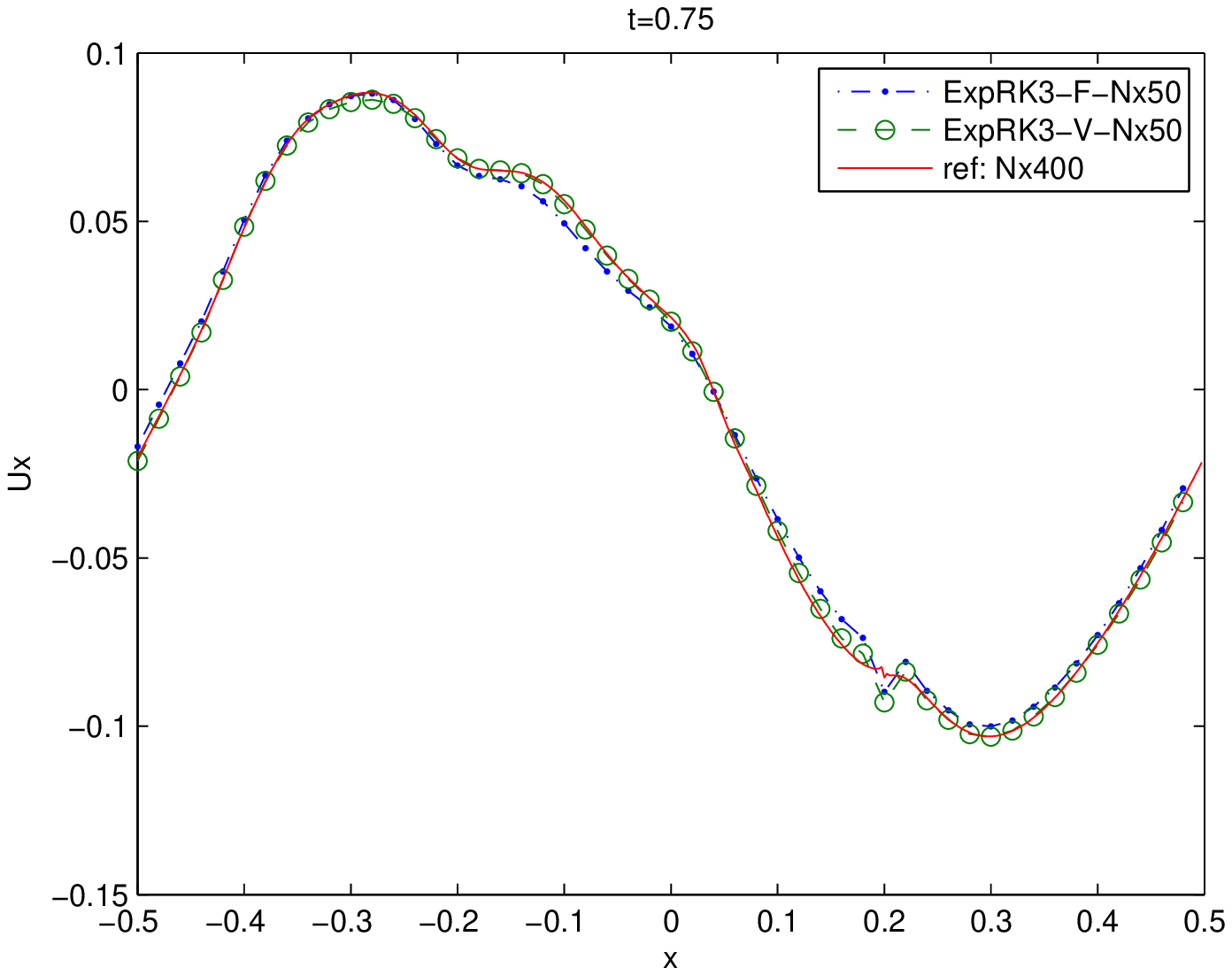}
     \end{subfigure}\begin{subfigure}\centering
        \includegraphics[width=0.5\textwidth,height=0.25\textheight]{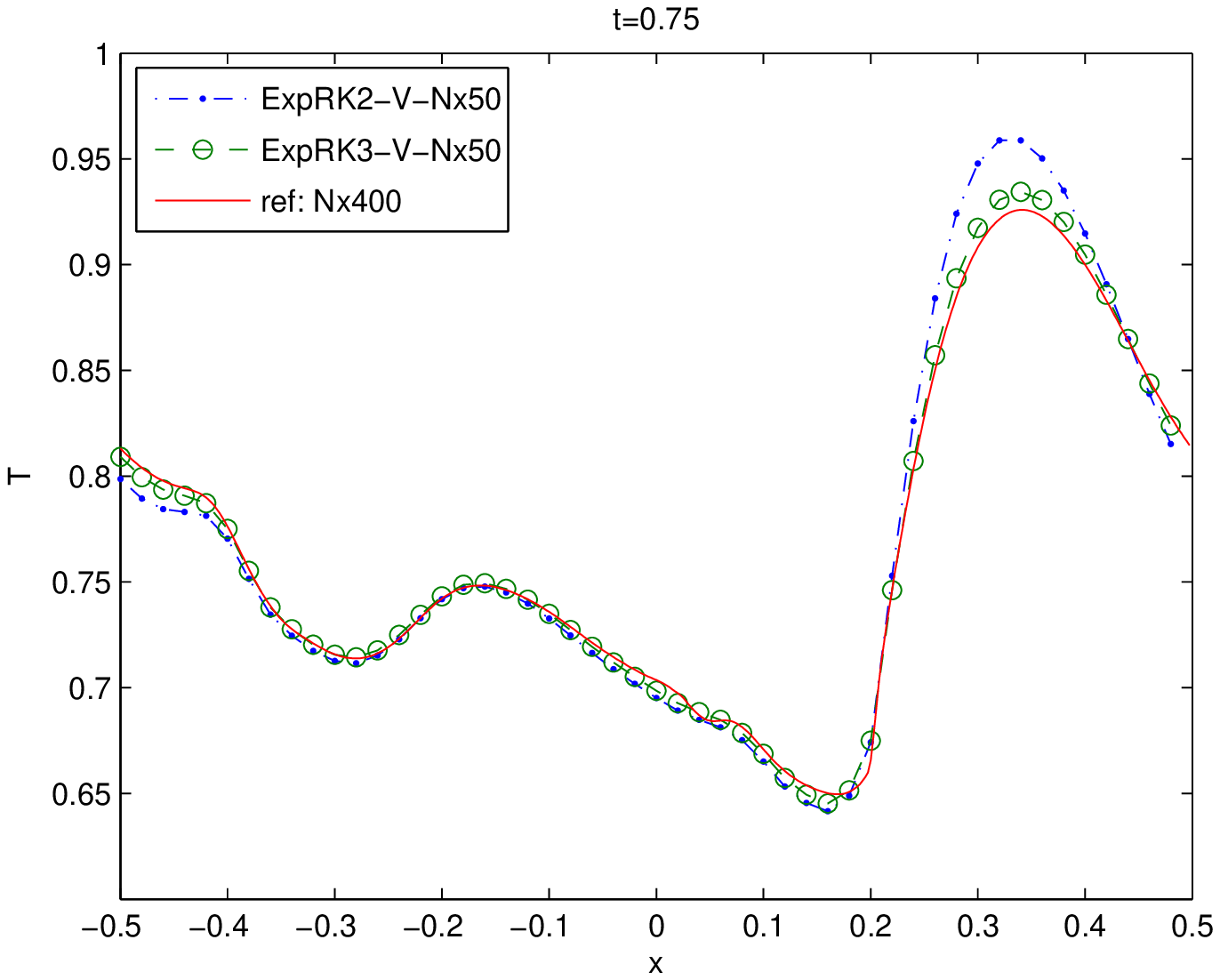}
    \end{subfigure}\begin{subfigure}\centering
        \includegraphics[width=0.5\textwidth,height=0.25\textheight]{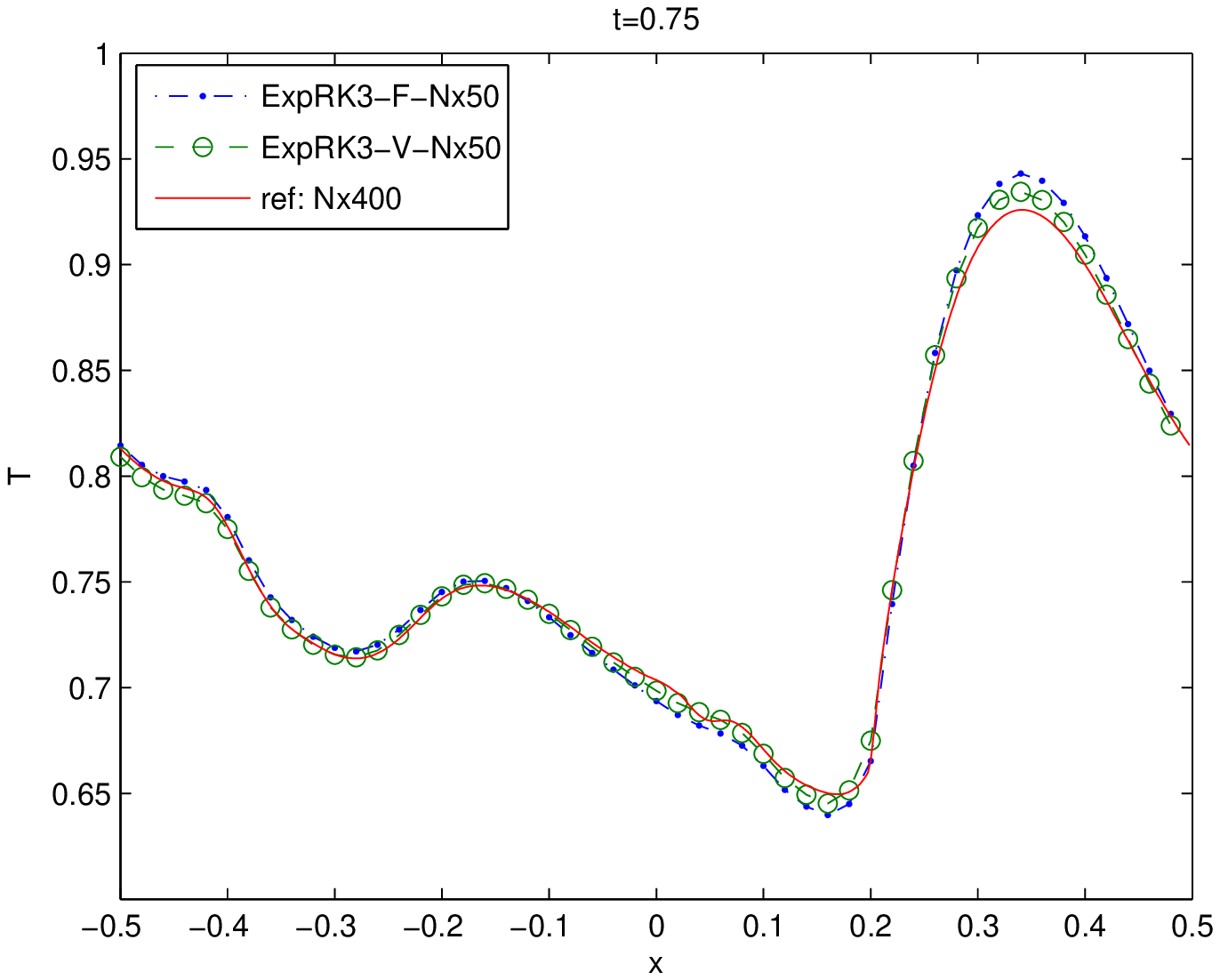}
    \end{subfigure}
   \caption{The left column shows comparison of RK2 and RK3 using the ExpRK-V. 
       The solid line is the reference solution with a very fine mesh in time and $\Delta x=0.005$, the dash line is given by RK3 and the 
       dotted line is given by RK2, both with $N_x=50$ points. The right 
       column compare two methods, both given by RK3, with the reference. The dash line 
       is given by ExpRK-V, and the dotted line is given by ExpRK-F. $N_x=50$ 
       for both. $h$ is chosen to satisfy CFL condition, in our case, the CFL 
   number is chosen to be $0.5$.}\label{fig_Ex3_ep3}
\end{figure}

\section{Conclusions and future developments}
In this paper we have presented a general way to construct high-order time discretization methods for the Boltzmann equations in stiff regimes which avoid the inversion of the collision operator. 
The main advantages compared to other methods presented in the literature is the capability to achieve high order uniformly with respect to the small Knudsen number and to originate monotone schemes thanks to the exponential structure of the coefficients. The approach presented here can be extended in principle to several other integro-differential kinetic equations where it is possible to identify a linear operator which preserves the asymptotic behavior of the system. For example in the case of the Landau equation this would involve the computation of the exact flow of the linear part, i.e. a matrix exponential, in the construction of the schemes.
We leave this possibility to future research.

\section*{Acknowledgements}
The first author would like to
thank Dr. G. Dimarco and Prof. S. Jin for stimulating discussions, and Dr.
Bokai Yan for providing the code of spectral method for the collision term.
\section{Appendix}
\subsection{Positivity of the mass density in ExpRK-V}
\begin{theorem}
    The method ExpRK-V defined by (\ref{scheme_M2}) gives positive $\rho$, and the negative part of $T$ is at most of order $O(h\varepsilon)$.
\end{theorem}
To prove this theorem, we firstly check the following lemma.
\begin{lemma}
    In each sub-stage, the distribution function $f^{(i)}$ and $M^{(i)}$ have
    the same first $d+2$ moments.
\end{lemma}
\begin{proof}
    We prove this for sub-stage $i$. Assume for $\forall j<i$, one has
    \begin{equation}\label{scheme_M2_assumption}
      \int\left(\begin{array}{c}1\\v\\\frac{v^2}{2}\end{array}\right)(f^{(j)}-M^{(j)})dv=0.
    \end{equation}
    Then, one could take moments of the first equation in the scheme
    (\ref{scheme_M2StepK}), and gets
    \begin{subequations}
    \begin{align}
        \int\left(\begin{array}{c}1\\v\\\frac{v^2}{2}\end{array}\right)(f^{(i)}-M^{(i)})e^{c_i\lambda}dv =&\int\left(\begin{array}{c}1\\v\\\frac{v^2}{2}\end{array}\right)\left(f^{n}-M^n\right)dv\label{scheme_M2_moments1}\\
                                                                                                          &+\displaystyle\sum^{i-1}_{j=1}a_{ij}\frac{\lambda}{\mu}e^{c_j\lambda}\int\left(\begin{array}{c}1\\v\\\frac{v^2}{2}\end{array}\right)\left( P^{(j)}-\mu M^{(j)}\right)dv \label{scheme_M2_moments2}\\
                                                                                                          &-\displaystyle\sum^{i-1}_{j=1}a_{ij}\frac{\lambda}{\mu}e^{c_j\lambda}\int\left(\begin{array}{c}1\\v\\\frac{v^2}{2}\end{array}\right)\left(\varepsilon v\cdot\nabla_xf^{(j)}-\varepsilon\partial_tM^{(j)}\right)dv\label{scheme_M2_moments3}
    \end{align}
\end{subequations}
(\ref{scheme_M2_moments1}) is zero for sure, (\ref{scheme_M2_moments2}) is zero
by definition of $P$ and (\ref{scheme_M2_assumption}).
(\ref{scheme_M2_moments3}) is zero because of the computation from
(\ref{scheme_ptM}). Thus it is obvious that $f^{(i)}$ and $M^{(i)}$ share the
same moments on each stage.
\end{proof}
With the previous lemma in hand, one could prove {Theorem 4}.
\begin{proof}
    As in the previous lemma, we only do the proof for sub-stage $i$. The
    final step can be dealt with in the same way. Rewrite the second equation
    of (\ref{scheme_M2StepK}) in Shu-Osher representation
    \begin{equation}
        \int \phi f^{(i)}dv=\sum_{j=1}^{i-1}\left(\alpha_{ij}\int\phi
    f^{(j)}dv+\beta_{ij}h\int\phi v\cdot\nabla_xf^{(j)}dv\right)
    \end{equation}
    This moment equation is the same as the equation on $\rho$ in the Euler
    system, and the classical proof for $\rho$ being positive for the Euler
    equation can just be adopted \cite{GottliebShu_TVDRK}. To check the
    positivity of $T$, one just need to make use of the last line of the
    moment equation, i.e.
    \begin{subequations}
    \begin{align}
        \int \frac{v^2}{2}f^{(i)}dv=&\sum_{j=1}^{i-1}\left(\alpha_{ij}\int\frac{v^2}{2}f^{(j)}dv+\beta_{ij}h\int\frac{v^2}{2} v\cdot\nabla_xf^{(j)}dv\right)\nonumber\\
                                   =&\sum_{j=1}^{i-1}\left(\alpha_{ij}\int\frac{v^2}{2}f^{(j)}dv+\beta_{ij}h\int\frac{v^2}{2} v\cdot\nabla_xM^{(j)}dv\right)\label{scheme_M2_pos1}\\
                                    &+h\sum_{j=1}^{i-1}\beta_{ij}\int\frac{v^2}{2} v\cdot\nabla_x\left(f^{(j)}-M^{(j)}\right)dv\label{scheme_M2_pos2}
\end{align}
\end{subequations}
(\ref{scheme_M2_pos1}) is exactly what one could get when computing for $E$ in the Euler
system: the form of $M$ closes it up. So the classical method to
prove that $E>\frac{\rho u^2}{2}$ in Runge-Kutta scheme could be used, and the
only thing new is from (\ref{scheme_M2_pos2}). However, as proved in the section
about AP, the difference between $f$ and $M$ is at most of $\varepsilon$, thus
(\ref{scheme_M2_pos2}) is of order $O(h\varepsilon)$.
\end{proof}

\subsection{$\left|P(f)-P(g)\right|\leq\left|f-g\right|$ in $d_2$ norm}
We adopt the results from \cite{TV_ProbMetrics}.
They denote $\mathrm{P}_2$ the collection of distributions $F$ such that
\begin{equation*}
    \int_{R^d}|v|^2dF(v)<\infty
\end{equation*}
A metric $\mathrm{d}_2$ on $\mathrm{P}_2$ is defined by
\begin{equation}
    \mathrm{d}_2(F,G)=\text{sup}_{\xi}\frac{\hat{f}(\xi)-\hat{g}(\xi)}{|\xi|^2}
\end{equation}
where $\hat{f}$ is the Fourier transform of $F$
\begin{equation*}
    \hat{f}(\xi)=\int e^{-i\xi\cdot v}dF(v)
\end{equation*}
One can transform the Boltzmann equation into its Fourier space and obtains\cite{PT_FourierBoltzmann,Bobylev_FourierMaxwellBoltzmann}
\begin{equation}\label{eqn_Boltzmann_Fourier}
    \partial_t\hat{f}(t,\xi)=\int_{S^2}B\left(\frac{\xi\cdot n}{|\xi|}\right)\left[\hat{f}(\xi^+)\hat{f}^(\xi^-)-\hat{f}^(\xi)\hat{f}(0)\right]dn
\end{equation}
where $\xi^\pm=\frac{\xi\pm|\xi|n}{2}$
\begin{theorem}
    $\mathrm{d}_2(P_f,P_g)<\mathrm{d}_2(f,g)$ for Maxwell molecules with cut-off collision kernel.
\end{theorem}
\begin{proof}
    For Maxwell molecule with cut-off collision kernel $\int B=S$. Thus
    $$\text{sup}|Q^-|=\text{sup}\left|\int Bf_*d\Omega dv_*\right|=\text{sup}|\rho S|<\infty.$$
    Considering $P=Q+\mu f=Q^++\left(\mu-Q^-\right)f$, it is enough to prove
    $\mathrm{d}_2(Q^+_f,Q^+_g)<C\mathrm{d}_2(f,g)$ for $C$ big enough. Given
    $$\hat{Q}^+_f=\int_{S^2}B\left(\frac{\xi\cdot n}{|\xi|}\right)\left[\hat{f}(\xi^+)\hat{f}^(\xi^-)\right]dn,$$
    one has
    $$\frac{\hat{Q}^+_f-\hat{Q}^+_g}{|\xi|^2}=\int_{S^2}B\left(\frac{\xi\cdot n}{|\xi|}\right)\left[\frac{\hat{f}(\xi^+)\hat{f}(\xi^-)-\hat{g}(\xi^+)\hat{g}(\xi^-)}{|\xi|^2}\right]dn$$
    From \cite{TV_ProbMetrics}, one gets
    $$\left|\frac{\hat{f}(\xi^+)\hat{f}(\xi^-)-\hat{g}(\xi^+)\hat{g}(\xi^-)}{|\xi|^2}\right|\leq\text{sup}\left|\frac{\hat{f}-\hat{g}}{|\xi|^2}\right|$$
    Thus, one has:
    \begin{equation*}
        \mathrm{d}_2(Q^+_f,Q^+_g)=\text{sup}_\xi\left|\frac{\hat{Q}^+_f-\hat{Q}^+_g}{|\xi|^2}\right|\leq S\sup\left|\frac{\hat{f}-\hat{g}}{|\xi|^2}\right|=S\mathrm{d}_2(f,g)
    \end{equation*}
\end{proof}

\begin{thebibliography}{10}

\bibitem{BLM}
{\sc M.~Bennoune, M.~Lemou, and L.~Mieussens}, {\em Uniformly stable numerical
  schemes for the {B}oltzmann equation preserving the compressible
  {Navier–Stokes} asymptotics}, Journal of Computational Physics, 227 (2008),
  pp.~3781--3803.

\bibitem{Bobylev_FourierMaxwellBoltzmann}
{\sc A.~V. {Bobylev}}, {\em {The {F}ourier transform method in the theory of
  the {B}oltzmann equation for {M}axwellian molecules}}, Akademiia Nauk SSSR
  Doklady, 225 (1975), pp.~1041--1044.

\bibitem{dimarco1}
{\sc P.~Degond, G.~Dimarco, and L.~Pareschi}, {\em The moment guided {Monte
  Carlo} method}, Int. J. Num. Meth. Fluids, 67 (2011), pp.~189-Ð213.

\bibitem{degond1}
{\sc P.~Degond, S.~Jin, and L.~Mieussens}, {\em A smooth transition model
  between kinetic and hydrodynamic equations}, J. of Comput. Phys., 209 (2005),
  pp.~665--694.

\bibitem{dimarco2}
{\sc G.~Dimarco and L.~Pareschi}, {\em Fluid solver independent hybrid methods
  for multiscale kinetic equations}, SIAM J. Sci. Comput., 32 (2010),
  pp.~603--634.

\bibitem{DP_ExpRK}
{\sc G.~Dimarco and L.~Pareschi}, {\em {E}xponential {Runge-Kutta} methods for
  stiff kinetic equations}, SIAM Journal on Numerical Analysis, 49 (2011),
  pp.~2057--2077.

\bibitem{dpimex}
{\sc G.~Dimarco and L.~Pareschi}, {\em Asymptotic preserving
  {Implicit-Explicit} {Runge-Kutta} methods for non linear kinetic equations},
  arXiv:1205:0882, preprint (2012).

\bibitem{Filbet}
{\sc F.~Filbet and S.~Jin}, {\em A class of asymptotic-preserving schemes for
  kinetic equations and related problems with stiff sources}, J. Comput. Phys.,
  229 (2010), pp.~7625--7648.

\bibitem{Filbet3}
{\sc F.~Filbet and S.~Jin}, {\em An asymptotic
  preserving scheme for the {ES-BGK} model of the {B}oltzmann equation}, SIAM
  J. Sci. Comput., 46 (2011), pp.~204--224.

\bibitem{toscani}
{\sc E.~Gabetta, L.~Pareschi, and G.~Toscani}, {\em Relaxation schemes for
  nonlinear kinetic equations}, SIAM J. Numer. Anal., 34 (1997),
  pp.~2168--2194.

\bibitem{GottliebShu_TVDRK}
{\sc S.~Gottlieb and C.-W. Shu}, {\em Total variation diminishing {Runge-Kutta}
  schemes}, Math. Comp, 67 (1998), pp.~73--85.

\bibitem{GST}
{\sc S.~Gottlieb, C.-W. Shu, and E.~Tadmor}, {\em Strong stability-preserving
  high-order time discretization methods}, SIAM Rev, 43 (2001), pp.~89--112.
  
\bibitem{HNW}
{\sc E.~Hairer, S.~P.~N/orsett, G.~Wanner}, {Solving ordinary differential equations I. Nonstiff problems}, Springer Series in Comput. Mathematics, Vol. 8, Springer-Verlag 1987, 3rd edition (2008).  

\bibitem{EX1}
{\sc M.~Hochbruck, C.~Lubich, and H.~Selhofer}, {\em Exponential integrators
  for large systems of differential equations}, SIAM J. Sci. Comput., 19
  (1998), pp.~1552--1574.

\bibitem{jinrev}
{\sc S.~Jin}, {\em Asymptotic preserving ({AP}) schemes for multiscale kinetic
  and hyperbolic equations: a review}, Lecture Notes for Summer School on
  ''Methods and Models of Kinetic Theory'' (M\&MKT), Porto Ercole (Grosseto,
  Italy), June 2010.

\bibitem{Lem}
{\sc M.~Lemou}, {\em Relaxed {micro–macro} schemes for kinetic equations},
  Comptes Rendus Mathematique, 348 (2010), pp.~455--460.

\bibitem{MZ}
{\sc S.~Maset and M.~Zennaro}, {\em Unconditional stability of explicit
  exponential {Runge-Kutta} methods for semi-linear ordinary differential
  equations}, Math. Comp., 78 (2009), pp.~957--967.

\bibitem{MP_FastSpectralCollision}
{\sc C.~Mouhot and L.~Pareschi}, {\em Fast algorithms for computing the
  {B}oltzmann collision operator}, Math. Comp., 75 (2006), pp.~1833--1852.

\bibitem{CPmc}
{\sc L.~Pareschi and R.~E. Caflisch}, {\em An implicit {Monte Carlo} method for
  rarefied gas dynamics \textrm{I}: The space homogeneous case.}, Journal of
  Computational Physics, 154 (1999), pp.~90--116.

\bibitem{PR}
{\sc L.~Pareschi and G.~Russo}, {\em Time relaxed {Monte Carlo} methods for the
  {B}oltzmann equation}, SIAM Journal on Scientific Computing, 23 (2001),
  pp.~1253--1273.

\bibitem{prrev}
{\sc L.~Pareschi and G.~Russo}, {\em Efficient asymptotic preserving
  deterministic methods for the {B}oltzmann equation}, AVT-194 RTO AVT/VKI, Models and Computational Methods for Rarefied Flows, Lecture Series held at the von Karman Institute, Rhode St. Gense, Belgium, 24 --28 January (2011).
  
\bibitem{Perthame_KineticScheme}
{\sc B.~Perthame}, {\em Boltzmann type schemes for gas dynamics and the entropy
  property}, SIAM Journal on Numerical Analysis, 27 (1990), pp.~1405--1421.

\bibitem{PT_FourierBoltzmann}
{\sc A.~Pulvirenti and G.~Toscani}, {\em The theory of the nonlinear
  {B}oltzmann equation for {M}axwell molecules in {F}ourier representation},
  Annali di Matematica Pura ed Applicata, 171 (1996), pp.~181--204.

\bibitem{RuSp} {\sc R.~J.~Spiteri and S.~J.~Ruuth}, {\em A new class of optimal high-order strong-stability preserving time discretization methods}, SIAM J. Numer. Anal., 40 (2002), pp.~469-Ð491.

\bibitem{Shu_WENOReview}
{\sc C.-W. Shu}, {\em Essentially non-oscillatory and weighted essentially
  non-oscillatory schemes for hyperbolic conservation laws}, ICASE Report No. 97-65, (1997).

\bibitem{ShuOsher_ENO}
{\sc C.-W. Shu and S.~Osher}, {\em Efficient implementation of essentially
  non-oscillatory shock-capturing schemes, ii}, Journal of Computational
  Physics, 83 (1989), pp.~32--78.

\bibitem{TK}
{\sc S.~Tiwari and A.~Klar}, {\em An adaptive domain decomposition procedure
  for {B}oltzmann and {E}uler equations}, Journal of Computational and Applied
  Mathematics, 90 (1998), pp.~223--237.

\bibitem{TV_ProbMetrics}
{\sc G.~Toscani and C.~Villani}, {\em Probability metrics and uniqueness of the
  solution to the {B}oltzmann equation for a {M}axwell gas}, Journal of
  Statistical Physics, 94 (1999), pp.~619--637.

\bibitem{YanJin_StrongAP}
{\sc B.~Y{an and S. Jin}}, {\em A successive penalty-based
  asymptotic-preserving scheme for kinetic equations},  preprint (2012).

\end{thebibliography}

\end{document}